\documentclass[a4paper]{amsart}
\usepackage{amsmath,amsthm,amssymb, amscd, amsfonts}
\usepackage[all]{xy}
\usepackage{leftidx}
\usepackage[latin1]{inputenc}
\usepackage{enumerate}

\usepackage[dvipsnames]{xcolor}

\theoremstyle{plain}

\newtheorem{theorem}{Theorem}[section]
\newtheorem{corollary}[theorem]{Corollary}

\newtheorem{proposition}[theorem]{Proposition}
\newtheorem{lemma}[theorem]{Lemma}
\newtheorem{question}[theorem]{Question}
\theoremstyle{definition}
\newtheorem{definition}[theorem]{Definition}

\newtheorem{example}[theorem]{Example}

\theoremstyle{remark}
\newtheorem{remark}[theorem]{Remark}

\numberwithin{equation}{section}\theoremstyle{plain}

\newcommand{\ootimes}{\overline{\otimes}}

\renewcommand{\1}{\textbf{1}}

\newcommand{\Hh}{\mathfrak{H}}

\newcommand\Opext{\operatorname{Opext}}

\newcommand{\A}{{\mathcal A}}
\newcommand{\Ag}{\mathfrak A}
\newcommand{\Bb}{{\mathfrak B}}

\newcommand{\B}{{\mathcal B}}
\newcommand{\C}{{\mathcal C}}
\newcommand{\D}{{\mathcal D}}

\newcommand{\Z}{{\mathcal Z}}
\newcommand{\Zz}{{\mathbb Z}}
\newcommand{\M}{\mathcal{M}}
\newcommand{\N}{\mathcal{N}}

\newcommand{\toto}{\longrightarrow}

\newcommand{\E}{{\mathcal E}}

\newcommand{\Aa}{\mathbb A}

\newcommand{\Ss}{\mathbb S}

\newcommand{\Rep}{\operatorname{Rep}}

\newcommand{\KER}{\mathfrak{Ker}}
\newcommand{\modd}{\mbox{-}\!\operatorname{mod}}
\newcommand{\comod}{\operatorname{comod}\!-\!}
\newcommand\CoRep[1]{\operatorname{comod}\!\mbox{-} #1}
\newcommand\rep{\operatorname{Rep}}
\newcommand\coend{\operatorname{coend}}
\newcommand\Aut{\operatorname{Aut}}

\newcommand\FPdim{\operatorname{FPdim}}

\newcommand\vect{\operatorname{Vec}}

\newcommand\id{\operatorname{id}}

\newcommand\Hker{\operatorname{Hker}}
\newcommand\Hcoker{\operatorname{Hcoker}}

\newcommand\End{\operatorname{End}}

\newcommand\op{\operatorname{op}}
\newcommand\Hom{\operatorname{Hom}}

\newcommand{\uno}{ \mathbf{1}}

\begin{document}


\title[EXACT SEQUENCES: FROM HOPF ALGEBRAS TO TENSOR CATEGORIES]{ON THE NOTION OF EXACT SEQUENCE: FROM HOPF ALGEBRAS TO TENSOR CATEGORIES}
\author{Sonia Natale}
\address{Facultad de Matem\'atica, Astronom\'\i a,  F\'\i sica y Computaci\' on.
Universidad Nacional de C\'ordoba. CIEM -- CONICET. Ciudad
Universitaria. (5000) C\'ordoba, Argentina}
\email{natale@famaf.unc.edu.ar
\newline \indent \emph{URL:}\/ http://www.famaf.unc.edu.ar/$\sim$natale}

\thanks{Partially supported by  CONICET and SeCYT--UNC}

\keywords{Hopf algebra; tensor category; fusion category; exact sequence}

\subjclass[2010]{18D10; 16T05}

\date{\today}

\begin{abstract} We present an overview of the notions of exact sequences of Hopf algebras and tensor categories and their connections. We also present some
examples illustrating their main features; these include simple fusion categories and a natural question regarding composition series of finite tensor categories. 
\end{abstract}

\maketitle

\setcounter{tocdepth}{2} \tableofcontents

\section{Introduction}

Perhaps the most natural way to understand or attempt a classification of a given algebraic structure is by decomposing it into 'simpler' structures. A fundamental example of this principle is provided by the theory of finite groups. Every finite simple group $G$ has a composition series whose factors are finite simple groups uniquely determined by $G$, up to permutations. In particular,  $G$ can be built up by means of a number of successive extensions of finite simple groups. In this context, both the classification of the simple structures -finite simple groups- and the classification of the possible extensions -governed by suitable cohomology theories- are hard questions that involve deep mathematical techniques. 

\medbreak In this paper we aim to present an overview of an analogous approach to the study of structures that generalize that of groups, namely, Hopf algebras and tensor categories. We discuss the notion of extension in each of these contexts and some of their main features. More precisely, we focus on extensions arising from exact sequences: these do not include certain extensions arising from group gradings on tensor categories which play an important role in the classification of certain classes of fusion categories \cite{ENO2}. 

\medbreak 
Very little is known about a possible approach to classify simple structures in these cases.  We discuss several simple examples that show some contrast with the theory of finite groups. 

\medbreak  
We start by recalling the notion of exact sequence of Hopf algebras in Section \ref{exact-hopf}. The main contributions towards this notion appeared in the work of G. I. Kac \cite{kac}, M. Takeuchi \cite{takeuchi}, W. Singer \cite{singer}, B. Parshall and J. P. Wang \cite{PW}, H.-J. Schneider \cite{schneider}, S. Majid \cite{majid}, N. Andruskiewitsch and J. Devoto \cite{AD}, \cite{andrus-ext}, Hofstetter \cite{hofstetter} and others.  

Exact sequences of Hopf algebras generalize exact sequences of groups. In a similar vein, the notion of normal subgroup and simple group have generalizations  to the notions of normal Hopf subalgebra and simple Hopf algebra, that we discuss in Subsections \ref{nhs} and \ref{subs-simple}. 

\medbreak 
As in the case of finite groups, every finite dimensional Hopf algebra $H$ has a \emph{composition series}: these are sequences of simple Hopf algebras $\Hh_1, \dots, \Hh_n$, called the \emph{factors} of the series, defined as follows \cite{AM}: If $H$ is simple, then $n = 1$ and $\Hh_1 = H$; 
if on the other hand, $H$ contains a  proper normal Hopf subalgebra $A$, and
$\Ag_1,
\dots, \Ag_m$, $\Bb_1, \dots,$ $\Bb_\ell$, are composition series of $A$ and of the quotient Hopf algebra $B = H/HA^+$,
respectively, then $n = m+\ell$ and
$$\Hh_i = \Ag_i, \quad \textrm{if } 1\leq i \leq m, \quad \Hh_i = \Bb_{i-m},
\quad \textrm{if } m < i \leq m+\ell.$$
Furthermore, a Jordan-H\" older theorem holds in the context of finite dimensional Hopf algebras \cite{jh-hopf}. This and related facts are discussed in Subsections \ref{subs-jh} and \ref{subs-semis}.

\medbreak In Section \ref{exact-rk1} we review the notion of exact sequence of tensor categories developed with A. Brugui\` eres  in \cite{tensor-exact}, \cite{indp-exact} and its relation with Hopf monads and commutative central algebras. We also discuss in this section a family of examples arising from so-called crossed actions of a matched pair of finite groups \cite{crossed-action} and recent classification results for extensions of a fusion category by the category of representations of a finite group obtained in \cite{char-crossed-action}.

\medbreak The notion of exact sequence of tensor categories discussed in Section \ref{exact-rk1} was extended to the notion of exact sequence  with respect to a module category by P. Etingof and S. Gelaki \cite{eg-emc}: In this sense, an exact sequence of (finite) tensor categories becomes an exact sequence with respect to a rank-one module category. We overview the definition of \cite{eg-emc} in Section \ref{s-emc} as well as its connection with exact  factorizations of tensor categories from \cite{gelaki},  \cite{mn}, \cite{schauenburg-kac}.

\medbreak We end the paper by presenting some examples that answer a number of natural questions regarding the behaviour of exact sequences of tensor categories. This is done in  Section \ref{examples-questions}. In contrast with the properties enjoyed by groups and Hopf algebras, these examples show in particular that:
\begin{itemize}
\item Fusion subcategories of index 2 need not be normal with respect to a module category (Subsection \ref{indice2});
\item There exist fusion categories of Frobenius-Perron dimension $p^aq^b$, where $p$ and $q$ are prime numbers, which are simple with respect to any module category (Subsection \ref{simple-paqb}).
\end{itemize}

\medbreak 
Regarding the question of formulating a Jordan-H\"older theorem for finite tensor categories (see \cite[Question 4.7]{icm}), we consider a natural generalization of the notion of composition series of finite dimensional Hopf algebra to the context of finite tensor categories: namely, we call a sequence of finite tensor categories $\C_1, \dots, \C_n$ -the factors of the series- a composition series of a finite tensor category $\C$ if, as before,  $n = 1$ and $\C_1 = \C$ if $\C$ is does not fit into any exact sequence with respect to a module category, while if $\C' \toto \C \toto \C'' \boxtimes \End(\M)$ is an exact sequence with respect to some $\C'$-module category $\M$ such that $\FPdim \C', \FPdim \C'' > 1$, then $n = m+\ell$ and
$\C_i = \C'_i$, $1\leq i \leq m$, $\C_i = \C''_{i-m}$, $m < i \leq m+\ell$, where  $\C'_1, \dots, \C'_m$, $\C''_1, \dots,$ $\C''_\ell$, are composition series of $\C'$ and $\C''$, respectively. 

\medbreak 
One of the main new contributions of this paper is a negative answer to the 
following question:

\begin{question}\label{jh-mc} Is it true that two composition series of a finite tensor category thus defined have the same factors up to a permutation?
\end{question}

Indeed we show that the answer to Question \ref{jh-mc} is negative even in the context of (braided non-degenerate) fusion categories. Hence composition series of fusion categories thus defined fail to satisfy a Jordan-H\"older theorem. This is done in Subsection \ref{subs-jh}; see Corollary \ref{cs-a6} and Remark \ref{bd-cs-a6}.

In spite of this fact, we mention that an analogue of the Jordan-H\" older theorem does hold for weakly group-theoretical fusion categories introduced in \cite{ENO2}. The definition of a composition series for this kind of category is given in terms of group equivariantizations and group graded extensions and the composition factors, which are Morita invariants, are finite simple groups \cite{jh-wgt}.  

\medbreak Finally, we include in Subsection \ref{preguntas} some questions that we believe are interesting in relation with the notions discussed previously.

\section{Preliminaries on Hopf algebras and tensor categories}

We shall work over an algebraically closed field $k$ of characteristic zero. 
A \emph{tensor category} over $k$ is 
a $k$-linear abelian category with finite dimensional $\Hom$
spaces and objects of finite length, endowed with a rigid monoidal category
structure, such that the monoidal product is $k$-linear in each variable
and the unit object is simple. 

A tensor category over $k$ is called  \emph{finite} if it is equivalent as a
$k$-linear category to the category of finite dimensional left modules over a
finite dimensional $k$-algebra. A \emph{fusion category over $k$} is a
semisimple finite tensor category. We refer the reader to \cite{EGNO} for a systematic study of tensor categories.

\medbreak 
Let $\C$ and $\D$ be tensor categories over $k$. A \emph{tensor functor} $F: \C \to \D$ is a $k$-linear exact (strong) monoidal functor $F$.

\medbreak
Let $G$ be a group. We shall denote by $\vect_{G}$ the tensor category of finite dimensional $G$-graded vector spaces and by $\Rep G$ the tensor category of finite dimensional $k$-linear representations of $G$.

\medbreak Let $H$ be a Hopf algebra over $k$. We shall indicate by $H^+$ the \emph{augmentation ideal} of $H$, defined as $H^+=\{x \in H:\, \varepsilon(x)=0\}$.
Except for Section \ref{exact-hopf}, all Hopf algebras we consider in this paper are assumed to have a bijective antipode (this is automatically true if the Hopf algebra is finite dimensional). Thus the categories $H\modd$ of finite dimensional (left) $H$-modules and $\CoRep H$ of finite dimensional (right) $H$-comodules are tensor categories over $k$. 

\medbreak
Let $\C$ be a tensor category. A \emph{left module category} over  $\C$, or \emph{$\C$-module category}, is a  $k$-linear abelian category $\mathcal{M}$ together with a $k$-linear bi-exact functor $$\ootimes:\mathcal{C}\times\mathcal{M}\rightarrow\mathcal{M},$$ and natural isomorphisms
$m_{X,Y,M}:(X\otimes Y)\ootimes M\rightarrow X\ootimes(Y\ootimes M)$, $u_M:\textbf{1}\ootimes M\rightarrow M$, $X, Y \in \C$, $M \in \mathcal M$, satisfying natural associativity and unitary conditions.

\medbreak 
Let $\M$ and $\N$ be $\C$-module categories. A $\C$-\emph{module functor} $\M \to \N$ is a $k$-linear functor $F:\M \to \N$ endowed with natural isomorphisms $$F (X \ootimes M) \to X \ootimes F(M),$$ for all $X \in \C$, $M \in \M$, satisfying appropriate conditions. A $\C$-module functor is called an equivalence of $\C$-module categories if it is an equivalence of categories. A module category $\M$ is called \emph{indecomposable} if it is not equivalent as a $\C$-module category to the direct sum of two nontrivial module categories.

\medbreak 
Suppose $\C$ is a finite tensor category. A finite left $\C$-module category $\M$ is \emph{exact} if for every projective object $P\in \C$ and for every $M\in\M$, $P\otimes  M$ is a projective object of $\M$.
Let $\M$ be an  indecomposable exact $\C$-module category. Then the category $\End_\C(\M)$ of right exact $\C$-module endofunctors of $\mathcal{M}$ is a finite tensor category.
A tensor category $\mathcal{D}$ is \emph{(categorically) Morita equivalent} to $\mathcal{C}$  if  $\mathcal{D}\cong \End_\C(\M)^{\text{op}}$ for some exact  indecomposable $\mathcal{C}$-module category $\mathcal{M}$.

\section{Exact sequences of Hopf algebras}\label{exact-hopf}

Let $H$ be a  Hopf algebra over $k$. The aim of this section is to give an account of some of the main features regarding the notion of exact sequences of Hopf algebras and present some examples. 

\medbreak 
Let  $\pi: H \to B$ be a Hopf algebra map. The subalgebras ${}^{co \pi}H$ and $H^{co \pi}$ of 
left and right $B$-coinvariants of $H$ are defined, respectively, by
\begin{align*}{}^{co \pi}H & = \{ h \in   H:\, (\pi \otimes \id)\Delta(h) = 1
\otimes h\}, \\
H^{co \pi} & = \{ h \in   H:\, (\id \otimes \pi)\Delta(h) = h \otimes
1\}.\end{align*}

\begin{definition}\label{def-exacta-hopf}(\cite{AD}.)
An \emph{exact sequence of Hopf algebras} is a sequence of Hopf algebra maps 
\begin{equation}\label{exacta} k \longrightarrow H' \overset{i}\longrightarrow H
\overset{\pi}\longrightarrow H'' \longrightarrow k,
\end{equation}
satisfying the following conditions:
\begin{itemize}
	\item[(a)] $i$ is injective and $\pi$ is surjective,
\item[(b)] $\ker \pi = Hi(H')^+$,
\item[(c)] $i(H') = {}^{co \pi}\!H$.
\end{itemize}

A Hopf algebra $H$ fitting into an exact sequence \eqref{exacta} is called an \emph{extension} of $H''$ by $H'$.
\end{definition}

Note that either of the conditions (b) or (c) in Definition \ref{def-exacta-hopf} implies that $\pi i = \epsilon 1: H' \to H''$. If $H$ is
faithfully
flat over  $H'$, then (a) and (b) imply (c).  Dually, if $H$ is
faithfully coflat over $H''$, then (a) and (c)
imply (b).

\medbreak 
We refer the reader to \cite{AD}, \cite{schneider}, \cite{takeuchi} for further details on the notion of an exact sequence. 

\medbreak 
Let \eqref{exacta} be an exact sequence of finite dimensional Hopf algebras. Then, as a consequence of the Nichols-Zoeller freeness theorem, it is \emph{cleft}, that is, the map $\pi:H \to B$ admits a convolution invertible $B$-colinear and $A$-linear section $B \to H$. 
This implies that $H$ isomorphic as a Hopf algebra to a bicrossed product $A\# B$ with respect to suitable compatible data. See \cite{schneider-nb}, \cite{AD}. In particular $$\dim H = \dim H' \dim H''.$$ 

In addition, $H$ is semisimple if and only if $H'$ and $H''$ are semisimple.

\begin{definition} We shall say that a Hopf algebra $H$ is  \emph{simple} if it does not fit into any exact sequence of Hopf algebras \eqref{exact-hopf} with $H' \ncong k$ and $H'' \ncong k$.	
\end{definition}

\subsection{Normality of Hopf subalgebras and Hopf algebra maps}\label{nhs} Consider the left and right adjoint actions of $H$ on itself defined, respectively, by
\begin{equation}\label{adj-act}
h. a = h_{(1)}a\mathcal S(h_{(2)}), \qquad a.h= \mathcal S(h_{(1)})ah_{(2)},
\end{equation}
for all $a, h \in H$. 

\medbreak Formulas \eqref{adj-act} generalize those defining the adjoint actions of a group. In this vein, the notion of a normal subgroup can be generalized to that of a normal Hopf subalgebra; that is, a Hopf subalgebra $K$ of $H$ is called \emph{normal} if it is stable under both action actions \eqref{adj-act}. 

\medbreak 
Dually, the left and right
adjoint coactions of $H$ are defined, respectively, by 
\begin{align*}
&\rho_\ell: H \to H \otimes H, \quad \rho_\ell(h) = h_{(1)}\mathcal S(h_{(3)}) \otimes h_{(2)},\\
&\rho_r: H \to H \otimes H, \quad \rho_r(h) = h_{(2)} \otimes h_{(1)}\mathcal S(h_{(3)}).
\end{align*}

Let $\pi: H \to H''$ be a Hopf
algebra map. The map $\pi$ is called \emph{normal} if the kernel $I$ of $\pi$ is a subcomodule for both adjoint coactions of $H$.

\medbreak
Let $H' \subseteq H$ be a normal Hopf subalgebra. Then $H(H')^+ = (H')^+H$
is a Hopf ideal of $H$ and the canonical map $H \to H/H(H')^+$ is a Hopf algebra map. If $H$ is faithfully flat over $H'$, then there is
an exact sequence of Hopf algebras $$k \longrightarrow H' \longrightarrow H
\longrightarrow H/H(H')^+ \longrightarrow
k.$$ Similarly, if $\pi: H \to H''$ is a surjective normal Hopf
algebra map, then ${}^{co \pi}H = H^{co \pi}$ is a Hopf subalgebra and if $H$ is faithfully coflat over $H''$, there is an exact sequence of Hopf algebras $$k \longrightarrow {}^{co \pi}H
\longrightarrow H \overset{\pi}{\longrightarrow} H'' \longrightarrow k.$$

\medbreak 
Suppose $H$ is a finite dimensional Hopf algebra. Then $H$ is simple if and only if $H$ contains no proper normal Hopf subalgebra if and only if it admits no proper normal quotient Hopf algebra.

Furthermore, a sequence of Hopf algebra maps $$k \longrightarrow H'
\overset{i}\longrightarrow H
\overset{\pi}\longrightarrow H'' \longrightarrow k$$ is an exact sequence if and only if the dual sequence 
$$k \longrightarrow (H'')^* \overset{\pi^*}\longrightarrow H^*
\overset{i^*}\longrightarrow (H')^* \longrightarrow k$$ is exact.
Therefore the notion of simplicity of a finite dimensional Hopf algebra is self-dual, that is, $H$ is simple if and only if $H^*$ is simple. 

\subsection{Example: Abelian exact sequences and matched pairs of groups}

A celebrated source of examples of Hopf algebra extensions arises from matched pairs of groups: these are called \emph{abelian extensions} and were introduced in the early work of G. I. Kac, W. Singer, S. Majid and M. Takeuchi \cite{kac}, \cite{singer}, \cite{majid}, \cite{takeuchi-mp}. We refer the reader to \cite{mk-ext}, \cite{ma-newdir} for a detailed study of the cohomology theory underlying an abelian exact sequence.

\begin{definition}
An exact sequence of finite dimensional Hopf algebras 
\begin{equation}\label{abel-hopf}
k \longrightarrow H'
\overset{i}\longrightarrow H
\overset{\pi}\longrightarrow H'' \longrightarrow k\end{equation} is called an \emph{abelian exact sequence} if $H'$ is commutative and $H''$ is cocommutative. 
\end{definition}
	
Since our base field $k$ is algebraically closed of characteristic zero, this means that $H' \cong k^\Gamma$ and $H'' \cong kG$, for some finite groups $\Gamma, G$. 

\medbreak 
The exactness of the sequence \eqref{abel-hopf} allows to endow the pair $(G, \Gamma)$ with the structure of a \emph{matched pair} of groups. That
is, \eqref{abel-hopf} gives rise to actions by permutations
$\Gamma \overset{\vartriangleleft}\longleftarrow \Gamma \times G
\overset{\vartriangleright}\longrightarrow G$ such that
\begin{equation}\label{matched}
s \vartriangleright xy  = (s \vartriangleright x) ((s
\vartriangleleft x) \vartriangleright y), \quad st
\vartriangleleft x  = (s \vartriangleleft (t \vartriangleright x))
(t \vartriangleleft x), \end{equation} for all $s, t \in \Gamma$,
$x, y \in G$.

\medbreak 
Given groups $G$ and $\Gamma$, the data of a pair
of compatible actions making $(G, \Gamma)$ into a matched pair of groups is equivalent to the data of a group $E$ together
with an \emph{exact factorization} into  subgroups (isomorphic to) $G$ and $\Gamma$: that is, a group $E$ such that $E = G \Gamma$ and $G \cap \Gamma = \{e\}$. In such situation, the relevant
actions $\vartriangleright$ and $\vartriangleleft$ are determined by the relations $$gx = (g \vartriangleright
x)(g \vartriangleleft x),$$ for every $x \in G$, $g \in \Gamma$.

\medbreak Fix a matched pair of finite groups $(G, \Gamma)$. Consider the left action of $G$ on $k^\Gamma$ defined by $(x.
f)(g) = f(g \vartriangleleft x)$, $f \in k^\Gamma$, and let
$\sigma: G \times G \to (k^*)^\Gamma$ be a normalized
2-cocycle. Dually, consider the right action of $\Gamma$ on
$k^G$ given by $(w.g)(x) = w(x \vartriangleright g)$, $w \in k^G$, and let
$\tau: \Gamma \times \Gamma \to (k^*)^G$ be a normalized
2-cocycle.

\medbreak 
Under appropriate compatibility conditions between $\sigma$ and
$\tau$, the vector space $H = k^\Gamma \otimes k G$  becomes a
(semisimple) Hopf algebra, denoted $H = k^\Gamma \,
{}^{\tau}\#_{\sigma}kG$, with the crossed product algebra
structure and the crossed coproduct coalgebra structure; see \cite{mk-ext},\cite{ma-newdir}. For all
$g,h\in \Gamma$, $x, y\in G$, we have
\begin{align}\label{mult} (e_g \# x)(e_h \# y) & = \delta_{g \vartriangleleft x, h}\, \sigma_g(x,
y) e_g \# xy, \\
\label{delta} \Delta(e_g \# x) & = \sum_{st=g} \tau_x(s, t)\, e_s
\# (t \vartriangleright x) \otimes e_{t}\# x,
\end{align}where $\sigma_s(x, y) = \sigma(x, y)(s)$ and $\tau_x(s, t) = \tau(s, t)(x)$, $s, t\in \Gamma$, $x, y\in G$.

\medbreak 
Let $i: k^\Gamma \to H = k^\Gamma \,
{}^{\tau}\#_{\sigma}kG$ and $\pi: H = k^\Gamma \,
{}^{\tau}\#_{\sigma}kG \to kF$ be the Hopf algebra maps defined by $i(f) = f \# e$, $\pi(f \# g) = \varepsilon(f) g$, $f \in k^\Gamma$, $g \in G$. The Hopf algebra $H$ fits into an abelian exact sequence 
$$k \toto k^\Gamma \overset{i}\toto H \overset{\pi}\toto kG \toto k.$$ 

Moreover, every Hopf algebra $H$
fitting into an abelian exact sequence \eqref{abel-hopf} is isomorphic to a bicrossed product $k^\Gamma \, {}^{\tau}\#_{\sigma}kF$ for an appropriate matched pair $(G, \Gamma)$ and compatible
actions and cocycles $\sigma$ and $\tau$. Equivalence classes of
such extensions associated to a fixed matched pair  $(G, \Gamma)$
form an abelian group $\Opext(k^\Gamma, kG)$, whose unit element
is the class of the \emph{split} extension $k^\Gamma \# kG$.

\medbreak 
An example of an abelian extensions is given by the Drinfeld double $D(G)$ of a finite group $G$: the Hopf algebra $D(G)$	fits into an exact sequence $$k \longrightarrow k^G \longrightarrow D(G)	\longrightarrow kG \longrightarrow k.$$ In the associated matched pair $(G, G)$, $\triangleleft: G \times G \to G$ is the adjoint action of $G$ on itself, while $\triangleright: G\times G \to G$ is the trivial action. 

\subsection{Jordan-H\" older theorem for finite dimensional Hopf algebras}\label{subs-jh}

The question of establishing an analogue of the Jordan-H\" older theorem of group theory for finite dimensional Hopf algebras was raised by N. Andruskiewitsch in \cite[Question
2.1]{andrus} and answered in \cite{jh-hopf}. 

\medbreak 
We start by recalling the followig definition given in \cite{AM}. 

\begin{definition}\label{def-compser} Let $H$ be a finite dimensional Hopf algebra over a field $k$. 
A \emph{composition series} of $H$ is a
sequence of finite dimensional simple Hopf algebras $\Hh_1, \dots, \Hh_n$
defined recursively as follows: If $H$ is simple, we let $n = 1$ and $\Hh_1 = H$.
If $k \subsetneq A \subsetneq H$ is a normal Hopf subalgebra, and
$\Ag_1,
\dots, \Ag_m$, $\Bb_1, \dots,$ $\Bb_\ell$, are composition series of $A$ and $B =
H/HA^+$,
respectively, then we let $n = m+\ell$ and
$$\Hh_i = \Ag_i, \quad \textrm{if } 1\leq i \leq m, \quad \Hh_i = \Bb_{i-m},
\quad \textrm{if } m < i \leq m+\ell.$$
The Hopf algebras $\Hh_1, \dots, \Hh_n$ are called the \emph{factors} of
the series.   The number $n$ is called the \emph{length} of the series.
\end{definition}

\medbreak 
Every finite dimensional Hopf algebra admits a composition series.
The following is an analogue of the Jordan-H\" older theorem for finite dimensional Hopf algebras:

\begin{theorem}\emph{(\cite[Theorem 1.2]{jh-hopf}.)}
Let  $\Hh_1, \dots, \Hh_n$ and $\Hh'_1, \dots, \Hh'_m$ be two composition series of $H$.
Then there exists a bijection $f: \{1, \dots, n\} \to \{1, \dots, m\}$ such that $\Hh_i \cong \Hh'_{f(i)}$ as Hopf algebras.
\end{theorem}

Let $H$ be a finite dimensional Hopf algebra and let
$\Hh_1, \dots,$ $\Hh_n$
be a composition series of $H$. The simple Hopf algebras $\Hh_i$, $1\leq i\leq
n$, are called the \emph{composition factors of $H$}. The number $n$ is 
called the \emph{length of $H$}.

\medbreak 
Let us consider for instance the case of abelian extensions. Let $G$ and $\Gamma$ be finite groups and let $H$ be
an abelian extension of $k^\Gamma$ by $kG$. Then the
composition factors of $H$ are the	group algebras of the composition factors of $G$ and the dual group algebras of	the composition factors of $\Gamma$. See \cite[Example 4.7]{jh-hopf}.

\medbreak 	
In particular,  if $G_1, \dots, G_n$ are  the composition	factors of the finite group $G$, then the composition factors of its Drinfeld double $D(G)$ are the	Hopf algebras  $k^{G_1},	\dots, k^{G_n}, kG_1, \dots, kG_n$. 

\subsection{Semisolvability and related questions}\label{subs-semis}
A \emph{lower subnormal
	series} of a Hopf algebra $H$ is
a series of Hopf subalgebras
\begin{equation}\label{lowerseries}k = H_{n}  \subseteq H_{n-1}
\subseteq  \dots
\subseteq H_1 \subseteq H_0 = H, \end{equation} such that $H_{i+1}$ is a
normal Hopf subalgebra of $H_i$, for all $i$. The \emph{factors} of the series
\eqref{lowerseries} are the quotient Hopf algebras $H_i/H_iH_{i+1}^+$, $i = 0, \dots, n-1$.

\medbreak 
Dually, an \emph{upper
	subnormal series} of $H$ is a series of surjective Hopf
algebra maps
\begin{equation}\label{upperseries}H = H_{(0)} \to H_{(1)}
\to \dots \to H_{(n)} = k,
\end{equation} such that $H_{(i+1)}$ is a normal quotient Hopf algebra of
$H_{(i)}$, for all $i = 0, \dots, n-1$. The \emph{factors} of
\eqref{upperseries} are the Hopf algebras ${}^{co \,
	H_{(i+1)}}H_{(i)} \subseteq H_{(i)}$, $i = 0, \dots, n-1$.

\medbreak A \emph{lower} (respectively \emph{upper}) \emph{composition series} of $H$ is a lower (respectively, upper) subnormal series which does not	admit any proper refinement. See \cite[Section 5]{jh-hopf}.

\medbreak
A Hopf algebra $H$ is called \emph{lower-semisolvable} (respectively, \emph{upper-semisolva\-ble}) if it admits a lower (respectively, upper) subnormal series whose factors are commutative or cocommutative \cite{MW}. We shall say that $H$ is \emph{semisolvable} if it is either lower or upper semisolvable. 

\medbreak 
Every semisolvable finite dimensional Hopf algebra is semisimple and cosemisimple.
On the other hand, every semisimple Hopf algebra of dimension $p^n$, $p$ a prime number, is semisolvable \cite{masuoka}, \cite{MW}. Moreover, if the group $G$ is nilpotent then any twisting of the group Hopf algebra $kG$ is semisolvable \cite{GN}.	

\medbreak 
In \cite{jh-hopf} we proved analogues of the Zassenhaus' butterfly	lemma and the Schreier's refinement theorem  for finite dimensional Hopf algebras and, following the lines of the classical proof  in	group theory, applied  them to prove an analogue of the	Jordan-H\" older theorem for lower and upper composition series of $H$. 

\medbreak 	
Thus the lower and upper composition factors of $H$ and its lower and upper lengths, which are also well-defined invariants of $H$, were introduced. In contrast with the case of the composition factors, the lower or upper	composition factors are not necessarily simple as Hopf algebras. This motivates the question of deciding if there is an	intrinsic characterization of the Hopf algebras that can arise as lower composition factors \cite[Question 5.7]{jh-hopf}. 
	
\medbreak 	
Some properties of lower and upper composition factors and their	relation with the composition factors were studied in \cite{jh-hopf}. Unlike for the case of the length, the	lower and upper lengths are not additive with respect to exact sequences and	they are not invariant under duality in general. 
	
\medbreak 	
Neither the composition factors nor the upper or
lower composition factors of a finite dimensional Hopf algebra $H$ are
categorical invariants of $H$. In other words, they are not invariant under
twisting deformations of $H$. In fact, there
exists a (semisimple) Hopf	algebra $H$ such that $H$ is simple as a Hopf algebra and $H$ is twist equivalent to	the group algebra of a solvable group $G$ (see Theorem \ref{twist-solv} below). In particular, the categories of	finite dimensional representations of $H$ and $G$ are equivalent fusion	categories. 

\subsection{Simple Hopf algebras}\label{subs-simple}
Recall that a finite dimensional Hopf algebra is simple if it contains no proper normal Hopf subalgebras. 
For instance, if $G$ is a finite simple group, then the group
algebra $kG$ and its dual $k^G$ are simple Hopf algebras.
Furthermore, in this case, any twisting deformation of $kG$ is
simple \cite{nik}. However, there are examples of solvable groups
that admit simple twisting deformations.

\medbreak Finite dimensional Hopf algebras with tensor equivalent categories
of representations are obtained from one another by a twisting
deformation. Properties of $H$ invariant under twisting are of
special interest because they depend only on the tensor category
$H\modd$.

\medbreak 
In the paper \cite{GN} we presented examples showing that the
notions of simplicity and (semi)solvability of a Hopf
algebra are \emph{not} twist invariants; that is, they are not
categorical notions.

\medbreak 
Let $p$, $r$ and $q$ be prime numbers such that $q$ divides $p-1$
and $r-1$. There is a family of supersolvable
groups $G$ of order $prq^2$ that can be deformed through a twist into
nontrivial simple Hopf algebras. 

\medbreak 
Let us recall this construction. Let $G_1 = \mathbb Z_p \rtimes \mathbb Z_q$ and $G_2 = \mathbb Z_r \rtimes \mathbb Z_q$ be the only nonabelian groups of
orders $pq$ and $rq$, respectively. Let $G = G_1 \times G_2$ and
let $S \cong \mathbb Z_q \times \mathbb Z_q$ be a
subgroup of $G$ order $q^2$. In particular, $G$ is supersolvable and
$Z(G) = 1$.

\medbreak Let $1\neq \omega \in H^2(\widehat S, k^*)$,  $J\in kG
\otimes kG$  the twist lifted from $S$ corresponding to  $\omega$.
Let also $H = (kG)^J$. Note that the cocycle $\omega$ is
nondegenerate. Also, $H$ is a noncommutative noncocommutative Hopf algebra of dimension $prq^2$.

\begin{theorem}\label{twist-solv}\emph{(\cite[Theorem 4.5]{GN}.)} The Hopf algebra $H$ is simple. \end{theorem}

The proof relied on the comparison of the (co)representation theory of
the given twistings \cite{eg-reptriang} with that of an extension
\cite{MW}. 

\medbreak 
Certain twists of the symmetric group
$\mathbb S_{n}$ on $n$ letters, $n \geq 5$, were also shown to be simple as Hopf algebras in \cite{GN}.

\medbreak 
As a consequence of Theorem \ref{twist-solv} the analogue of Burnside's $p^aq^b$-Theorem for finite groups does not hold for semisimple Hopf algebras. 

\medbreak 
Theorem \ref{twist-solv} provides the smallest example of a noncommutative noncocommutative semisimple Hopf algebra which is simple: this appears in dimension $36$ as a twisting of $D_3\times D_3$. This turns out to be the only simple example in dimension $< 60$:

\begin{theorem}\emph{(\cite[Theorem 1]{ss-lowdim}.)}
Every semisimple Hopf algebra of dimension less than 60 is semisolvable up to a cocycle twist.
\end{theorem}

\medbreak The previously mentioned results on simplicity of twisting deformations provide us with three examples of non-commutative non-cocommutative semisimple Hopf algebras of dimension 60 which are simple as Hopf algebras. The
first two are the Hopf algebras $\A_0$ and $\A_1 \simeq \A_0^*$
constructed by D. Nikshych \cite{nik}. We have $\A_0 = (k\mathbb
A_5)^J$, where $J \in k\mathbb A_5 \otimes k\mathbb A_5$ is an
invertible twist lifted from a nondegenerate $2$-cocycle in a
subgroup of $\mathbb A_5$ isomorphic to $\Zz_2 \times \Zz_2$.

The third example is the self-dual Hopf algebra $\B$ constructed
in \cite{GN}. In this case $\B = (kD_3\otimes kD_5)^J$, where $J$
is an invertible twist also lifted from a nondegenerate
$2$-cocycle in a subgroup of $D_3\times D_5$ isomorphic to $\Zz_2
\times \Zz_2$.

\medbreak 
As coalgebras, these examples are isomorphic to direct sums of
full matric coalgebras, as follows: \begin{align}\label{coalg-a1}
\A_1 & \simeq k \oplus M_3(k)^{(2)}\oplus M_4(k)\oplus M_5(k),
\\ \label{coalg-a0} \A_0 & \simeq k^{(12)}\oplus M_4(k)^{(3)}, \\ \label{coalg-b}  \B & \simeq
k^{(4)}\oplus M_2(k)^{(6)}\oplus M_4(k)^{(2)}. \end{align} As for
the group-like elements, we have $G(\A_0) \simeq \mathbb A_4$ and
$G(\B) \simeq \Zz_2 \times \Zz_2$.

\medbreak It was shown in
\cite[Theorem 9.12]{ENO2} that $\Rep \A_0 \simeq \Rep \mathbb A_5$
is the only fusion category of dimension $60$ which contains no proper fusion
subcategories. 

\medbreak The following theorem was shown in \cite{sixty}, answering Question 2.4 in \cite{andrus}.

\begin{theorem}\emph{(\cite[Theorem 1.4]{sixty}.)} Let $H$ be a nontrivial semisimple Hopf algebra of dimension	$60$. Suppose $H$ is simple. Then $H$ is isomorphic to $\A_0$ or	to $\A_1$ or to $\B$. \end{theorem}

\section{Exact sequences of tensor categories}\label{exact-rk1}

Let $\C$ and $\C''$ be tensor categories over $k$. 
A tensor functor $F: \C \to \C''$ is called \emph{dominant} (or \emph{surjective}) if every object of $\C''$ is a subobject of $F(X)$ for some object $X$ of $\C$. If $\C''$ is a finite tensor category, then $F$ is dominant if and only if every object of $\C''$ is a subquotient of $F(X)$ for some object $X$ of $\C$; see \cite[Lemma 2.3]{eg-emc}.

\medbreak A tensor functor $F: \C \to \C''$ is called \emph{normal}
if for every object $X$ of $\C$, there
exists a subobject $X_0 \subset X$ such that $F(X_0)$ is the
largest trivial subobject of $F(X)$.

\medbreak 
For a tensor functor $F: \C \to \C''$, let $\KER_F$ denote the tensor
subcategory
$F^{-1}(\langle \uno \rangle) \subseteq \C$  of objects $X$ of
$\C$ such that $F(X)$ is a trivial object of $\C''$. If the functor $F$ has a right adjoint $R$, then $F$ is normal if and only if
$R(\uno)$ belongs to $\KER_F$ \cite[Proposition 3.5]{tensor-exact}.

\begin{definition}\label{exact-rk1}Let $\C', \C, \C''$ be tensor categories over $k$. An \emph{exact
	sequence of tensor categories} is a sequence of
tensor functors
\begin{equation}\label{exacta-fusion}\C' \overset{f}\toto \C \overset{F}\toto \C'',
\end{equation}
such that the tensor functor $F$ is dominant and normal and the tensor functor
$f$ is a full embedding whose  essential image  is $\KER_F$ \cite{tensor-exact}.

\medbreak 
If $\C$ fits into an exact sequence \eqref{exacta-fusion}, we say that $\C$ is an \emph{extension of $\C''$ by $\C'$}.
\end{definition}

Every exact sequence of tensor categories 
\eqref{exacta-fusion}
defines a fiber functor on the kernel $\C'$: $$\omega=\Hom(\1,Ff) : \C' \to \vect.$$

\medbreak 
The \emph{induced Hopf algebra} $H$ of the exact sequence \eqref{exacta-fusion} is defined as
\begin{equation}
H = \coend (\omega);
\end{equation}
see \cite[Subsection 3.3]{tensor-exact}. Thus there is an equivalence of tensor categories $\C' \cong \comod H$ such that the following diagram of tensor functors is commutative:
$$\xymatrix{\C' \ar[r]^{\cong \quad } \ar[rd]_{\omega} & \comod H \ar[d]^{U}\\ & \vect.}$$

The induced Hopf algebra $H$ of \eqref{exacta-fusion} is finite-dimensional if and only if the tensor functor $F$ admits adjoints \cite[Proposition 2.15]{tensor-exact}. Hence if $\C'$ and $\C''$ are finite tensor categories, then so is $\C$.

\medbreak The fiber functor $\omega$ corresponds to a (rank one) $\C'$-module category structure on $\vect$, that we shall denote by $\M$ (see \cite[Proposition 4.1]{ostrik}). Assume that $H$ is finite dimensional. As a consequence of \cite[Theorem 4.2]{ostrik} we obtain the following relation between $\M$ and the induced Hopf algebra:
$$(\C')^*_\M \cong H\modd.$$

\subsection{Examples from exact sequences of Hopf algebras}\label{subs-exactahopf}
All Hopf algebras considered in this subsection are assumed to have a bijective antipode.

\medbreak 
Let $f : H_1 \to H_2$ be a Hopf algebra map. Then $f$ induces by restriction tensor functors 
$$f_* : \CoRep{H_1} \to \CoRep{H_2}, \qquad f^* : H_2\modd \to H_1\modd.$$

The  functor  $f_* : \CoRep{H_1} \to \CoRep{H_2}$ is dominant if and only if the functor $\underline{\quad} \square_{H_2}H_1 :  \operatorname{Comod}\!\mbox{-}{H_2} \to \operatorname{Comod}\!\mbox{-}{H_1}$ is faithful. On the other hand, if $f$ is injective and $H_2$ is finite-dimensional, then $f^* : H_2\modd \to H_1\modd$ is dominant. See \cite[Lemma 2.11]{tensor-exact}. 

\medbreak
Recall that the map $f$ has a kernel and a cokernel in the category of Hopf algebras, called the \emph{categorical kernel} and the \emph{categorical cokernel} of $f$, and defined, respectively, in the form
\begin{align*}
\Hker(f)&=\{h \in H_1 \mid h_{(1)} \otimes f(h_{(2)}) \otimes h_{(3)} = h_{(1)} \otimes 1 \otimes h_{(2)} \}\\
\Hcoker(f)&=H_2/H_2 f(H_1^+)H_2.
\end{align*}

The following proposition describes the kernels of the tensor functors induced by a Hopf algebra map.

\begin{proposition}\emph{(\cite[Lemma 2.10]{tensor-exact}.)}
Let $f : H \to H'$ be a morphism of Hopf algebras over a field,  and let
$$f_*:  \CoRep{H}\to \CoRep{H'}, \qquad  f^* : H'\modd \to H\modd$$ be the tensor functors induced by $f$.
Then the following hold:
\begin{enumerate}\item
$\KER_{f_*}= \CoRep{\Hker(f)}$. Moreover the tensor functor $f_*$ is normal if and only if
$H^{co H'}=\,^{co H'}\!\!H$, and in this case $\Hker(f)=H^{co H'}$.
\item $\KER_{f^*}= \Hcoker(f)\modd$. Moreover the tensor functor $f^*$ is normal if $f(H)$ is a normal Hopf subalgebra of $H'$, and in this case  $\Hcoker(f)={H'/H'f(H^+)}$.
\end{enumerate}
\end{proposition}

We shall say that an exact sequence of Hopf algebras 
\begin{equation}\label{st-exact}
k \toto H' \overset{i}\toto  H \overset{\pi}\toto H'' \toto k,
\end{equation}
is \emph{strictly exact} if $H$ is faithfully coflat over $H''$. 

\begin{remark}\label{rmk-st-exact} Notice that, since the antipode of $H$ is bijective, $H$ is right faithfully coflat over $H''$ if and only if it is left faithfully coflat over $H''$. Moreover, if this is the case, then $\pi$ is normal if and only if $^{co \pi}H = H^{co \pi}$. See \cite{takeuchi}.
	
\medbreak Observe in addition that the antipode of  $H$ induces an anti-isomorphism of algebras $\mathcal S: {}^{co \pi}H \to H^{co \pi}$ with inverse $\mathcal S^{-1}: H^{co \pi} \to {}^{co \pi}H$. In particular, if $i(H') = {}^{co \pi}H$, then ${}^{co \pi}H$ (being a Hopf subalgebra) is stable under the antipode and therefore $^{co \pi}H = H^{co \pi}$.	
\end{remark}

\medbreak Assume that the sequence \eqref{st-exact} is strictly exact. Then, as can be seen from the facts recalled in Remark \ref{rmk-st-exact}, it is a \emph{strictly exact sequence of Hopf algebras} in the sense of \cite{schneider}, that is, the following condiditions hold:

\begin{enumerate}[(a)]
	\item $\pi$ is a normal Hopf algebra map,
\item $H$ is right faithfully coflat over $H''$,
\item $i(H') = \Hker \pi$.
\end{enumerate}

These conditions are furthermore equivalent to the following:

\begin{enumerate}[(a')]
\item $H'$ is a normal Hopf subalgebra of $H$,
\item  $H$ is right faithfully flat over $H'$,
\item $H'' = \Hcoker i$.
\end{enumerate}

In this way we obtain:

\begin{theorem}\emph{(\cite[Proposition 3.9]{tensor-exact}.)}
Every strictly exact sequence of Hopf algebras \eqref{st-exact} gives rise to an exact sequence of tensor categories
$$\CoRep{H'} \overset{i_*}\toto \CoRep{H} \overset{\pi_*}\toto \CoRep{H''}.$$ If in addition $H$ is finite-dimensional, it also gives rise to an exact sequence of tensor categories:
$$H''\modd \overset{\pi^*}\toto H\modd \overset{i^*}\toto H'\modd.$$ 
\end{theorem} 

For instance, an exact sequence of groups $1\toto G' \toto G \toto G''\toto  1$ gives rise to an exact sequence of tensor categories
\begin{equation*}\vect_{G'} \toto \vect_{G} \toto \vect_{G''},\end{equation*} and if $G$ is finite, to an exact sequence
\begin{equation*}\rep G'' \toto \rep G \toto \rep G'.\end{equation*}


\subsection{Extensions and normal Hopf monads}
We refer the reader to \cite{BV}, \cite{BLV} for the notion of Hopf monad on a monoidal category.

\medbreak 
Let $F: \C \to \C''$ be a tensor functor between tensor categories and \emph{assume that $F$ admits a left adjoint $G$}. Then the composition $FG: \C'' \to \C''$  is a $k$-linear right exact Hopf monad on $\C''$, and $\C$ is tensor equivalent to the category $(\C'')^T$ of $T$-modules in $\C''$.

\medbreak The functor $F$ is dominant if and only if $T$ is faithful \cite[Proposition 4.1]{tensor-exact}.

\medbreak 
A Hopf monad $T$ on a tensor category $\C''$ is called \emph{normal} if $T(\1)$ is a trivial object of $\C''$. When $T$ is the Hopf monad corresponding to $F$ as above, the normality of $T$ is equivalent to the normality of $F$ \cite[Proposition 4.6]{tensor-exact}. 

\medbreak 
Normal faithful Hopf monads classify extensions of tensor categories in view of the following theorem:

\begin{theorem}\emph{(\cite[Theorem 4.8]{tensor-exact}.)}\label{thm-ex-nhm} Let	$\C'$, $\C''$ be tensor categories and assume that $\C'$ is finite. Then the following	data are equivalent:
\begin{enumerate}
\item A normal faithful $k$-linear right exact Hopf monad $T$ on $\C''$, with induced Hopf algebra $H$, endowed with a tensor equivalence $K :\C' \cong \CoRep{H}$;
\item An extension $\C' \to \C \to \C''$ of $\C''$ by $\C'$.
\end{enumerate}
\end{theorem}

Under the correspondence established by Theorem \ref{thm-ex-nhm}, the induced Hopf algebra of the exact sequence $\C' \to \C \to \C''$ is identified with the Hopf monad of the restriction of $T$ to the trivial subcategory of $\C''$.

\subsection{Perfect exact sequences and central commutative algebras}\label{pft-exact}

Let $\C$ and $\C''$ be tensor categories. A tensor functor $F: \C \to \C''$ is called \emph{perfect} if it admits an exact right adjoint \cite[Subsection 2.1]{indp-exact}.

\medbreak 
Every tensor functor $F: \C \to \C''$ between finite tensor categories $\C$ and $\C''$ is perfect \cite[Lemma 2.1]{char-crossed-action}.

\medbreak Let $F: \C \to \C''$ be a dominant perfect tensor functor and let $R: \C'' \to \C$ be a right adjoint of $F$ and let $A = R(\1)$.  Then there exists a half-braiding $\sigma$ on $A$ such that $(A, \sigma)$ is a  commutative algebra in the Drinfeld center $\Z(\C)$ of $\C$, which satisfies $\Hom_\C(\1, A) \cong k$. The algebra $(A, \sigma)$ is called the \emph{induced central algebra of $F$} \cite[Section 6]{tensor-exact}.

\medbreak 
The category $\C_A$ or right $A$-modules in $\C$ becomes a tensor category with tensor product $\otimes_A$ and unit object $A$, equipped with a tensor functor $F_A = ? \otimes A: \C \to \C_A$.
Furthermore, there is an equivalence of tensor categories $\kappa: \C'' \to \C_A$ such that the following diagram of tensor functors is commutative up to a monoidal natural isomorphism
$$\xymatrix{\C \ar[r]^{F} \ar[rd]_{F_A} & \C'' \ar[d]^{\kappa}\\ & \C_A.}$$

\medbreak 
For instance, let $H$ be a finite dimensional Hopf algebra. Then the induced central algebra $(A,\sigma)$ of the forgetful functor $\comod H \to \vect_k$ can be described as follows: As an algebra in $\comod H$,	$A=H$ with the right regular coaction $\Delta: A \to A \otimes H$. For any right $H$-comodule $V$, the half-braiding $\sigma_V: A \otimes V \to V \otimes A$ is determined by the right adjoint action of $H$ in the form	$$\sigma_{V}(h \otimes v) =  v_{(0)} \otimes S(v_{(1)})\, h\, v_{(2)}, \quad h \in H, v \in V.$$ See \cite[Example 6.3]{tensor-exact}.

\medbreak Let $\C$ be a tensor category. We shall say that an algebra $(A, \sigma)$ in the center $\Z(\C)$ is a \emph{central algebra of} $\C$ if $(A, \sigma)$ is the induced central algebra of some dominant perfect tensor functor. Thus a central algebra of $\C$ is a commutative algebra $(A, \sigma) \in \Z(\C)$ such that $\Hom_\C(\1, A)$ is one-dimensional. The dominant tensor functor corresponding to $(A, \sigma)$ is the functor $- \otimes A: \C \to \C_A$, whose right adjoint is the forgetful functor $\C_A \to \C$.

\begin{definition}
An exact sequence sequence of tensor categories $\C' \toto \C \overset{F}\toto \C''$ is called a \emph{perfect exact sequence} if $F$ is a perfect tensor functor. 
\end{definition}

Every exact sequence of finite tensor categories is a perfect exact sequence.
In addition, in every perfect exact sequence $\C' \toto \C \overset{F}\toto \C''$, the kernel $\C'$ is a finite tensor category; see \cite[Proposition 3.15]{tensor-exact}. Thus the induced Hopf algebra of a perfect exact sequence is always finite dimensional.

\medbreak 
Let $(\E): \; \C' \toto \C \toto \C''$ be a perfect exact sequence and let  $(A, \sigma) \in \Z(\C)$ be the induced central algebra of $F$. Then $(A, \sigma)$ is  \emph{self-trivializing}, that is, $A \otimes A$ is a trivial object of $\C_A$. Let $\langle A \rangle$ be the smallest abelian subcategory of $\C$ containing
$A$ and stable under direct sums, subobjects and quotients. Then $F_A: \C \to \C_A$ is a normal dominant tensor functor with $\KER_{F_A} = \langle A \rangle$. Moreover, $(\E)$ is equivalent to the exact sequence
\begin{equation}\langle A \rangle \toto \C \overset{F_A}\toto	\C_A.
\end{equation}
See \cite[Subsection 6.2]{tensor-exact}.


\subsection{Examples from finite groups: equivariantization and crossed extensions by matched pairs}

Let $G$ be a finite group and let $\C$ be a tensor category. A monoidal functor $\rho: \underline G^{\op} \to  \Aut_k(\C)$ is called a right action of $G$ on $\C$ by $k$-linear autoequivalences. 

\medbreak 
The \emph{equivariantization} of $\C$ under the action $\rho$ is the $k$-linear abelian category $\C^G$ whose objects are pairs $(X, r)$, where $X$ is an object of $\C$ and $r = (r^g)_{g \in G}$ is a collection of isomorphisms $r^g:\rho^g(X)
\to X$, $g \in G$, such that for all $g, h \in G$,
\begin{equation}\label{deltau} r^g \rho^g(r^h) = r^{hg} (\rho^{g,
	h}_2)_X,\end{equation}
where $\rho_2^{g, h}:\rho^g\rho^h \to \rho^{hg}$ is the monoidal structure of $\rho$, and 
$r^e{\rho_0}_X=\id_X$. A morphism $f: (X, r) \to
(Y, r')$ is a morphism $f: X \to Y$ in $\C$ such that $fr^g = {r'}^g\rho^g(f)$, for all $g \in G$.

\medbreak The forgetful functor $F: \C^G \to \C$ gives rise to a perfect exact sequence of tensor categories
\begin{equation}\label{equiv-exact}
\rep G \toto \C^G \overset{F}\toto \C,
\end{equation} with induced Hopf algebra $H \cong k^G$ \cite[Subsection 5.3]{tensor-exact}. An exact sequence of tensor categories equivalent to \eqref{equiv-exact} is called an \emph{equivantization exact sequence}.

\medbreak A \emph{$G$-grading} of $\C$ is a decomposition $\C = \bigoplus_{g \in G}\C_g$ into abelian subcategories $\C_g$, called the  \emph{homogeneous components} of the grading, such that $$\C_g \otimes \C_h \subseteq \C_{gh},$$ for all $g, h \in G$. 
The neutral homogeneous component $\C_e$ is a tensor subcategory of $\C$. A $G$-grading $\C = \bigoplus_{g \in G}\C_g$ is faithful if $\C_g \neq 0$, for all $g \in G$.

\medbreak
Every finite tensor category $\C$ has a faithful universal grading $\C = \bigoplus_{u \in U(\C)}\C_u$, with neutral homogeneous component $\C_e$ equal to the adjoint subcategory $\C_{ad}$, that is, the smallest tensor
Serre subcategory of $\C$ containing the objects $X \otimes X^*$, where $X$ runs over the simple objects of $\C$. The group $U(\C)$ is called the \emph{universal grading group} of $\C$.

The upper central series of $\C$, $\dots \subseteq \C^{(n+1)} \subseteq \C^{(n)} \subseteq \dots \subseteq \C^{(0)} = \C$, is defined as $\C^{(0)} = \C$ and $\C^{(n+1)} = (\C^{(n)})_{ad}$, for all $n \geq 0$.
A tensor category $\C$ is called \emph{nilpotent} if there exists some $n \geq 0$ such that $\C^{(n)} \cong \vect$.
See \cite[Section 3.5 and 4.14]{EGNO}, \cite{gel-nik}.

\medbreak 
Let $\C$ be a tensor category and let $(G, \Gamma)$ be a matched pair of  groups. A $(G, \Gamma)$-\emph{crossed action} on $\C$  consists of the following data:

\begin{itemize}\item  A $\Gamma$-grading $\C = \bigoplus_{s \in \Gamma}
	\C_s$.
	
	\item  A right action of $G$ by $k$-linear autoequivalences $\rho:
	\underline{G}^{\op} \to \Aut_k(\C)$ such that
	\begin{equation}\label{rho-partial} \rho^g(\C_s) = \C_{s \lhd g},\quad \forall
	g\in G, \, s\in \Gamma,\end{equation}
	
	\item A collection of natural isomorphisms $\gamma = (\gamma^g)_{g\in G}$:
	\begin{equation}\label{gamma}\gamma^g_{X, Y}: \rho^g(X \otimes Y) \to
	\rho^{t\rhd g}(X) \otimes \rho^g(Y), \quad X \in \C, \, t\in \Gamma,\,  Y \in
	\C_t, \end{equation}
	
	\item A collection of isomorphisms $\gamma^g_0: \rho^g(\uno) \to \uno$, $g \in
	G$.
\end{itemize}

These data are subject to the commutativity of the following diagrams:

\begin{itemize}\item[(a)] For all $g \in G$, $X \in \C$, $s, t \in \Gamma$, $Y
	\in \C_s$, $Z
	\in \C_t$,
	\begin{equation*}
	\xymatrix @C=0.6in @R=0.45in{
		\rho^g(X \otimes Y \otimes Z) \ar[rr]^{\gamma^g_{X\otimes Y, Z}}
		\ar[d]_{\gamma^g_{X, Y \otimes Z}} & & \rho^{t \rhd g}(X \otimes Y) \otimes
		\rho^g(Z) \ar[d]^{\gamma^{t\rhd g}_{X, Y} \otimes \id_{\rho^g(Z)}} \\
		\rho^{st \rhd g}(X) \otimes \rho^g(Y \otimes Z)\ar[rr]_{\id_{\rho^{st \rhd
					g}(X)}
			\otimes \gamma^g_{Y, Z} \qquad }&& \rho^{s \rhd (t \rhd g)}(X) \otimes \rho^{t
			\rhd g}(Y) \otimes \rho^g(Z)}
	\end{equation*}
	
	\item[(b)] For all $g \in G$, $X \in \C$,
	\begin{equation*}\xymatrix @C=0.6in @R=0.45in{
		\rho^g(X) \otimes \rho^g(\uno)  \ar[dr]_{\id_{\rho^g(X)} \otimes \gamma^g_0
			\quad } & \ar[l]_{\quad \gamma^g_{X, \uno}} \rho^g(X) \ar[d]^=
		\ar[r]^{\gamma^g_{\uno, X} \quad } & \rho^g(\uno) \otimes \rho^g(X)
		\ar[dl]^{\gamma^g_0 \otimes \id_{\rho^g(X)}} \\
		& \rho^g(X) }\end{equation*}

	\item[(c)] For all $g, h \in G$, $X \in \C$, $s\in \Gamma$, $Y \in \C_s$,
	\begin{equation*}\xymatrix@C=0.6in @R=0.45in{
		\rho^g\rho^h(X\otimes Y) \ar[dd]_{\rho^g(\gamma^h_{X, Y})} \ar[r]^{{\rho_2}_{X
				\otimes Y}^{g, h}} & \rho^{hg}(X \otimes Y) \ar[d]^{\gamma^{hg}_{X, Y}} \\
		& \rho^{s \rhd hg}(X) \otimes \rho^{hg}(Y) \\
		\rho^g(\rho^{s\rhd h}(X) \otimes \rho^h(Y)) \ar[r]_{\gamma^g_{\rho^{s\rhd h}(X),
				\rho^h(Y)}}  & \rho^{(s\lhd h)\rhd g}\rho^{s\rhd h}(X) \otimes \rho^g\rho^h(Y)
		\ar[u]_{{\rho_2}_X^{(s\lhd h) \rhd g, s \rhd h} \otimes
			{\rho_2}_Y^{g, h}} }\end{equation*}

	\item[(d)] For all $g, h \in G$,
	\begin{equation*}\xymatrix@C=0.6in @R=0.45in{
		\rho^g\rho^h(\uno) \ar[d]_{\rho^{g}(\gamma^h_0)} \ar[r]^{(\rho_2^{g, h})_{\uno}}
		& \rho^{hg}(\uno) \ar[d]^{\gamma^{hg}_0} \\
		\rho^g(\uno) \ar[r]_{\gamma^{g}_0} & \uno}
	\end{equation*}
	
	\item[(e)] For all $X \in \C$, $s \in \Gamma$, $Y \in \C_s$,
	\begin{equation*}\xymatrix@C=0.6in @R=0.45in{
		X\otimes Y \ar[dr]_{{\rho_0}_X \otimes {\rho_0}_Y} \ar[r]^{{\rho_0}_{X\otimes
				Y}} & \rho^e(X\otimes Y) \ar[d]^{\gamma^{e}_{X, Y}} \\
		& \rho^e(X)\otimes \rho^e(Y)}
	\xymatrix@C=0.6in @R=0.45in{
		\uno \ar[dr]_= \ar[r]^{{\rho_0}_\uno} & \rho^e(\1) \ar[d]^{\gamma^e_0}
		\\
		& \uno}
	\end{equation*}
\end{itemize}

We say that a  tensor category $\C$ is a \emph{$(G, \Gamma)$-crossed tensor category} if it is equipped with a $(G, \Gamma)$-crossed action.
In this case there is a canonical  exact sequence of tensor categories
\begin{equation}\label{c-abelian}\Rep G \toto \C^{(G, \Gamma)} \overset{F}\toto
\C,\end{equation} 
where $\C^{(G, \Gamma)}$ is the tensor category defined as follows: As a $k$-linear category, it is the equivariantization $\C^G$ of $\C$ under the action $\rho$, while the tensor product is defined in the form $(X, r) \otimes (Y, r') = (X \otimes Y, \tilde
r)$, where  ${\tilde r}^g$, $g \in G$, is the composition
$$\bigoplus_{s \in \Gamma} \rho^g(X \otimes Y_s)
\overset{\oplus_s\gamma^g_{X, Y_s}}\toto \bigoplus_{s \in \Gamma} \rho^{s \rhd
	g}(X) \otimes \rho^g(Y_s) \overset{\oplus_s r^{s \rhd g} \otimes {r'}^g_s}\toto
\bigoplus_{s \in \Gamma} X \otimes Y_{s\lhd g} = X \otimes Y,$$
for $Y = \bigoplus_{s \in \Gamma}Y_s$, $Y_s \in \C_s$. The functor $F: \C^{(G, \Gamma)} \to \C$ is the forgetful functor $F(V, (r^g)_{g \in G}) = V$. See \cite[Theorem 6.1]{crossed-action}.

\medbreak 
We call $\C^{(G, \Gamma)}$  a  \emph{$(G, \Gamma)$-crossed extension of $\C_e$}.
Thus, a $(G, \Gamma)$-crossed extension is a unified formulation of equivariantizations and group graded extensions. In fact,
suppose that $\rho: \underline{G}^{op} \to \Aut_{\otimes}(\C)$ is an action by tensor auto-equivalences of a tensor category $\C$. Then the equivariantization $\C^G$ is a $(G, \{e\})$-crossed extension of $\C$, where $\{e\}$ is the trivial group endowed with the trivial actions $\lhd: \{e\} \times G \to G$ and $\rhd: \{e\} \times G \to \{e\}$.	
On the other hand, if $\C$ is a tensor category graded by a group $\Gamma$, then $\C$ is a $(\{e\}, \Gamma)$-crossed extension of $\C_e$ in a similar way.

\medbreak 
Further examples of $(G, \Gamma)$-crossed extensions are the categories of representations of abelian extensions of Hopf algebras. Indeed for abelian every exact sequence $k \toto k^\Gamma \toto H \toto kG \toto k$, the category  $H\modd$ is a $(G, \Gamma)$-crossed extension of $\vect$. See \cite[Subsection 8.2]{crossed-action}.

\medbreak 
There exist $(G, \Gamma)$-crossed extensions that cannot be built up by means of equivariantizations or group graded extensions.
For instance, let $n \geq 5$ be an odd integer and let $H = k^{\Aa_{n-1}} \# kC_n$, $n \geq 5$,  be the bicrossed product associated to the matched pair $(C_n, \Aa_{n-1})$ arising from the exact factorization $\Aa_{n} = \Aa_{n-1} C_n$ of the alternating group $\Aa_n$,  where $C_n = \langle (12\dots n)\rangle$ \cite[Section 8]{mk-ext}. 
	
As shown in \cite[Example 4.3]{char-crossed-action}, if $G$ is a nontrivial finite group, then $H\modd$ is not equivalent to a $G$-equivariantization or to a $G$-graded extension of any fusion category $\C$. 

\subsection{Abelian exact sequences of tensor categories and matched pairs of groups}

\begin{definition} (\cite[Definition 5.1]{char-crossed-action}.) An exact sequence of tensor categories is an \emph{abelian exact sequence} if its induced Hopf algebra $H$ is finite dimensional and commutative.
\end{definition}

Equivalently, an abelian exact sequence is an exact sequence of the form $$\Rep G \toto \C \toto \C'',$$ such that the induced
tensor functor $\omega: \Rep G \to \vect$ is monoidally isomorphic to the forgetful functor or, in other words, such that the corresponding rank-one module category $\M$ is equivalent to the trivial rank-one module category of $\Rep G$. 

\medbreak 
Examples of abelian exact sequences of tensor categories arise from equivariantization under the action of a finite group on a tensor category and also from Hopf algebra extensions of the form $$k \toto k^G \toto H \toto H'' \toto k,$$ where $G$ is a finite group.

\medbreak Let $(G, \Gamma)$ be a matched pair of finite groups and let $\C$ be a $(G, \Gamma)$-crossed tensor category. The induced Hopf algebra of the associated exact sequence of tensor categories
\begin{equation}\Rep G \toto \C^{(G, \Gamma)} \overset{F}\toto
\C\end{equation} 
is $H \cong k^G$, so that \eqref{c-abelian} is an abelian exact sequence. 
The main result of \cite{char-crossed-action} says that crossed extensions by matched pairs do in fact exhaust  the class of abelian exact sequences of finite tensor categories:

\begin{theorem}\emph{(\cite[Theorem 1.1]{char-crossed-action}.)} Let $G$ be a finite group and let $$(\E): \; \Rep G \toto \C \toto \D$$ be an abelian exact sequence of finite tensor categories. Then there exists a finite group $\Gamma$ endowed with mutual actions by permutations $\rhd: \Gamma \times G \to G$, $\lhd: \Gamma  \times G \to \Gamma$ and a $(G, \Gamma)$-crossed action on $\D$ such that $(\E)$ is equivalent to the exact sequence $\Rep G \toto \D^{(G, \Gamma)} \toto \D$.
\end{theorem}

\section{Exact sequences of finite tensor categories with respect to a module category}\label{s-emc}

Let $\C$ and $\D$ be finite $k$-linear abelian categories. Their \textit{Deligne tensor
	product} is a finite tensor category denoted $\C \boxtimes \D$ endowed with a
functor $\boxtimes: \C \times \D \to \C \boxtimes \D$ exact in both variables
such that for any $k$-bilinear right exact functor $F : \C \times \D \to \A$,
where $\A$ is a $k$-linear abelian category, there exists a unique right exact
functor
$\tilde F : \C \boxtimes \D \to \A$ such that $\tilde F \circ \boxtimes = F$.
Such a category exists and it is unique up to equivalence. In fact, if $\C
\cong A\textrm{-mod}$ and $\D \cong B\textrm{-mod}$, for some finite dimensional
$k$-algebras $A$ and $B$, then $\C \boxtimes \D \cong (A\otimes
B)\textrm{-mod}$. See \cite{deligne-t}.
The tensor product of two finite (multi-)tensor categories $\C$ and $\D$ is again a
finite tensor category and if $\C$ and $\D$ are (multi-)fusion categories, then so is
$\C \boxtimes \D$.

\medbreak 
Let $Y, Z$ be objects of $\C$ and $\D$, respectively, and let us denote $Y \boxtimes Z = \boxtimes(Y, Z)$. Then $\Hom_{\C \boxtimes \D}(Y_1 \boxtimes Z_1, Y_2 \boxtimes Z_2) \cong \Hom_{\C}(Y_1, Y_2) \otimes \Hom_{\D}(Z_1, Z_2)$, for all $Y_1, Y_2 \in \C$, $Z_1, Z_2 \in \D$.
The simple objects of $\C \boxtimes \D$ are exactly those of the form $Y \boxtimes Z$, where $Y$ is a simple object of $\C$ and $Z$ is a simple object of $\D$.

\medbreak 
Let $\C' \subseteq \C$ and $\C''$ be finite tensor categories and let $\M$ be an
exact  indecomposable left $\C'$-module category. In particular, $\M$ is finite. Let $\End(\M)$ denote the
category of $k$-linear right exact endofunctors of $\M$, which is a monoidal
category with tensor product given by composition of functors and unit object $1_\M:\M \to \M$. Let also $i: \C'
\to \C$ denote the inclusion functor.

\begin{definition} 
An exact sequence of tensor categories \textit{with respect to $\M$} is a
sequence of exact monoidal functors
\begin{equation}\label{exact-m} \quad \C' \overset{i}\longrightarrow \C
\overset{F}\longrightarrow \C'' \boxtimes \End(\M),\end{equation}
such that $F$ is dominant, $\C' = \KER_F$ coincides with the subcategory of $\C$ mapped to $\End(\M)$ under $F$ and, $F$ is normal in the sense that for every object $X$ of $\C$, there exists a subobject $X_0$ of $X$ such that $F(X_0)$ is the largest subobject of $F(X)$ contained in $\End(\M)$. 

A tensor category  $\C$ fitting into an exact sequence \ref{exact-m} with respect to $\M$ is called an \emph{extension of $\C''$ by $\C'$ with respect to $\M$}.
\end{definition}

\medbreak The notion of exact sequence with respect to a module category was introduced in \cite{eg-emc} and it generalizes the notion of exact sequence of \cite{tensor-exact}. Indeed, suppose that \eqref{exact-m} is an exact sequence of finite tensor categories with respect to $\M$. Then $\FPdim \C = \FPdim \C' \FPdim \C''$. Moreover this condition characterizes the exactness of \eqref{exact-m} under the assumptions that $F$ is dominant and $\C' \subseteq \KER_F$ \cite[Theorem 3.6]{eg-emc}.

\medbreak 
Consider an exact sequence of finite tensor categories \begin{equation}\label{wrt-rk1}
\C' \toto \C \toto \C'',
\end{equation} as introduced in Section \ref{exact-rk1}. Then \eqref{wrt-rk1} induces a fiber functor $\omega: \C' \to \vect$, thus making $\M = \vect$ into a rank-one $\C'$-module category. In this way \eqref{wrt-rk1} becomes an exact module category \emph{with respect to the rank-one module category $\M$}. 

\medbreak 
The Deligne tensor product $\C'' \boxtimes \C'$ of two finite tensor categories gives rise to an exact sequence \eqref{exact-m} with respect to any exact indecomposable $\C'$-module category $\M$, where $F: \C'' \boxtimes \C' \to \C'' \boxtimes \End(\M)$ is the natural dominant monoidal functor. \cite{eg-emc}. 

\medbreak
The notion of exact sequence with respect to a module
category is self-dual in the following sense: Let $\N$ be an indecomposable exact $\C''$-module category. Then \eqref{exact-m} induces an exact sequence with respect to $\N$:
\begin{equation}\label{exact-n} \quad (\C'')^*_\N \overset{F^*}\longrightarrow \C^*_{\N\boxtimes \M}
\overset{i^*}\longrightarrow (\C')^*_\M \boxtimes \End(\N).\end{equation}

\medbreak 
Suppose that \eqref{exact-m} is an exact sequence with respect to $\M$. Observe that if $\C$ is a fusion category, then $\C'$ and $\C''$ are fusion categories and $\M$ is a finite semisimple $\C'$-module category. 

Assume conversely that $\C'$ and $\C''$ are fusion categories. In particular  $\M$ is a finite semisimple module category over $\C'$.  By \cite[Theorem 3.8]{eg-emc} $\C$ is also a fusion category.
In this case the monoidal category $\End(\M) \cong \leftidx{^\vee}{\!}{}\M \boxtimes \M$ is a multifusion category. Here, $\leftidx{^\vee}{\!}{}\M$ is the right $\C'$-module category such that $\leftidx{^\vee}{\!}{}\M = \M$ and $M \bar\otimes X = {}^*X \otimes M$, $X \in \C'$, $M \in \M$; that is, if $\M \cong \C'_A$ is the category of right $A$-modules in $\C'$ for some indecomposable algebra $A \in \C'$, then $\leftidx{^\vee}{\!}{}\M \cong {}_A\C'$.

\begin{lemma}\label{quotient-pt} Let $\C' \overset{i}\longrightarrow \C
\overset{F}\longrightarrow \C'' \boxtimes \End(\M)$ be an exact sequence with respect to $\M$. If $\C$ is a pointed fusion category, then so is $\C''$. \end{lemma}

\begin{proof} By assumption every simple object $X$ of $\C$ is invertible, that is, $X \otimes X^* \cong \1$.
Let $1_\M = \bigoplus_{i\in I} 1_i$ be a decomposition of the unit object $1_\M$ of $\End(\M)$ into a direct sum of simple subobjects $1_i$, $i \in I$. Then $1_i \otimes 1_j \cong \delta_{i, j} 1_i$, and $1_i^* \cong {}^*1_i \cong 1_i$, for all $i, j \in I$; see \cite[Section 4.3]{EGNO}.
	
Let $Y$ be a simple object of $\C''$. For each fixed $i \in I$, $Y \boxtimes 1_i$ is a simple object of $\C'' \boxtimes \End(\M)$. Since the functor $F: \C \to \C'' \boxtimes \End(\M)$ is dominant, then $Y \boxtimes 1_i$ is a direct summand of $F(X)$ for some simple object $X$ of $\C$.
Therefore $(Y \boxtimes 1_i) \boxtimes (Y \boxtimes 1_i)^* \cong (Y\otimes Y^*) \boxtimes 1_i$ is a direct summand of $F(X) \otimes F(X)^*$. On the other hand, 
$$F(X \otimes X^*) \cong F(\1) \cong \1_{\C \boxtimes \End(\M)} \cong \1 \boxtimes 1_\M \cong \bigoplus_{i\in I} \1 \boxtimes 1_i.$$
Therefore the only simple constituent of $Y \otimes Y^*$ is the trivial object $\1$ of $\C''$. This implies that $Y \otimes Y^* \cong \1$, that is, $Y$ is invertible. Since the simple object $Y$ was arbitrary, this shows that $\C''$ is pointed, as claimed.
\end{proof}

\subsection{Exact sequences and exact factorizations}

Let $\C$ be a fusion category. Recall from \cite{gelaki} that an $\C$ is endowed with an \emph{exact factorization} into a product of two fusion subcategories $\A$ and $\B$ if the following conditions hold: 

\begin{itemize}
	\item $\C$ coincides with the full abelian subcategory spanned by direct summands of $X \otimes Y$, $X \in \A$, $Y \in \B$,
	\item $\A \cap \B = \langle \1\rangle$.
\end{itemize}  

If this holds, we write $\C = \A \bullet \B$. 

\medbreak By \cite[Theorem 3.8]{gelaki}, $\C = \A \bullet \B$ if and only if every simple object $Z$ of $\C$ admits a decomposition $Z \cong X \otimes Y$, for unique (up to isomorphism) simple objects $X \in \A$, $Y \in \B$.

Thus the notion of an exact factorization can be formulated as a category equivalence as follows: $\C$ has an exact factorization $\C = \A \bullet \B$ if and only if the tensor product of $\C$ induces an equivalence of $k$-linear categories $$\otimes : \A \boxtimes \B \to \C.$$ 

\medbreak For example, let $E$ be a finite group. Exact factorizations of $E$ 
correspond exactly to exact factorizations of the category $\vect_E$. In fact, if $\Gamma$ and $G$ are subgroups of $E$, then $E = \Gamma G$ is an exact factorization of $E$ if and only if $\vect_E = \vect_\Gamma \bullet \vect_G$ is an exact factorization of $\vect_E$. 

\medbreak 
Exact factorizations and extensions are related as follows: Every exact sequence $\C' \overset{i}\longrightarrow \C
\overset{F}\longrightarrow \C'' \boxtimes \End(\M)$ with respect to $\M$ induces an exact factorization of the dual fusion category $\C^*_{\C'' \boxtimes \M}$:
\begin{equation}\C^*_{\C'' \boxtimes \M} = \C'' \bullet (\C')^*_\M.
\end{equation}
 
Conversely, every exact factorization $\C = \A \bullet \B$ of a fusion category $\C$ induces an exact sequence, with respect to any indecomposable $\A$-module category $\N$
\begin{equation}\A^*_{\N} \toto \C^*_{\B \boxtimes \N} \toto \B \boxtimes \End(\N).
\end{equation}
In particular, $\C$ fits into an exact sequence with respect to $\N = \A$: 
\begin{equation}\A \toto \C \toto \B \boxtimes \End(\A).\end{equation} See \cite[Theorem 4.1]{gelaki}.

\subsection{Exact factorizations arising from exact sequences of finite tensor categories.}

Consider an exact sequence of Hopf algebras 
\begin{equation}\label{e-hopf}
k \toto H' \toto H \toto	H'' \toto k,
\end{equation} such that $H'$ is finite dimensional. Then $H$ is free	as	a left (or right) module over $H'$ and in particular the sequence is cleft \cite[Theorem 2.1 (2)]{schneider}. By \cite[Proposition 3.9]{tensor-exact} the exact sequence \eqref{e-hopf} gives rise to an exact sequence of tensor categories
\begin{equation}\label{sec-comod} \CoRep H' \to \CoRep H \to \CoRep H''.
\end{equation}

\medbreak 	
Observe that $H'$ is naturally an algebra in $\CoRep H$.	By a result of \cite{schauenburg-kac}, the tensor category ${}_{H'}(\CoRep H)_{H'}$ is equivalent to the category of comodules over a certain coquasibialgebra $((H')^* \bowtie H'', \varphi)$,  where $\varphi$ is a so-called \emph{Kac 3-cocycle} associated to the exact sequence \eqref{e-hopf} (see \cite[Section 6]{schauenburg-kac}). The coquasibialgebra $((H')^* \bowtie H'', \varphi)$, termed a \emph{generalized product coquasibialgebra} in \cite{schauenburg-kac}, has an exact factorization into its sub-coquasibialgebras $(H')^*$ and $H''$.

\medbreak A generalization of this feature to the context of exact sequences of finite tensor categories appeared in \cite{mn}. Let $(\E): \C' \toto \C \toto \C''$ be an exact sequence of finite tensor categories  and let $(A, \sigma) \in \Z(\C)$ be its induced central algebra. So that $\C'' \cong \C_A$ and the exact sequence $(\E)$ is equivalent to the exact sequence $\langle A \rangle  \toto \C \overset{F_A}\toto \C_A$. See Subsection \ref{pft-exact}.

\medbreak 
Let $_A\C_A$ be the tensor category of $A$-bimodules in $\C$. The category $\C'' \cong \C_A$ is an indecomposable exact $\C$-module category and there are equivalences of tensor categories $$\C_{\C''}^*  \cong {}_A\C_A, \qquad H\modd \cong {}_A(\comod H)_A \cong {}_A{\C'}_{\!\!A},$$
where $H$ is the induced Hopf algebra of $(\E)$.

\medbreak
It was shown in \cite[Proposition 7.3]{mn} that there is an equivalence of $k$-linear categories  \begin{equation}\label{ef-bn} \C_{\C''}^* \cong (H\modd) \boxtimes \C'',\end{equation} thus, in the sense of the formulation explained at the beginning of this section, an exact factorization of $\C_{\C''}^*$.

\medbreak 
More precisely, under the identifications $H\modd \cong {}_A(\comod H)_A \cong {}_A{\C'}_{\!\!A} \subseteq {}_A\C_A$ and $\C'' \cong \C_A \subseteq {}_A\C_A$,  the tensor product functor $\otimes_A :{}_A{\C'\!}_A \boxtimes \C_A \to {}_A\C_A$ induces an equivalence of $k$-linear categories
\begin{equation}\label{f-gral} (H\modd) \boxtimes \C_A \toto {}_A\C_A.
\end{equation} 

Notice that in the situation of the exact sequence \eqref{sec-comod}, there is an equivalence of tensor categories $$(\CoRep H)_{\CoRep H''}^* \cong {}_{H'}(\CoRep H)_{H'}.$$ See \cite[Example 7.4]{mn}.

\section{Examples and open questions}\label{examples-questions}

We begin this section by discussing some examples that answer a number of natural questions regarding the behaviour of exact sequences in relation with known facts about exact sequences of groups and Hopf algebras.

\medbreak Let $\C$ be a finite tensor category. Let also $\C'$ be a tensor subcategory of $\C$ and  $\M$ an indecomposable exact $\C'$-module category. We shall say that $\C'$ is \emph{normal in $\C$ with respect to $\M$} if there exist a tensor category $\C''$ and an exact sequence of tensor categories with respect to $\M$: $$\C' \toto \C \toto \C''\boxtimes \End(\M).$$

A tensor subcategory $\C'$ will be called \emph{normal in $\C$} if it is normal with respect to a rank-one module category, that is, if there exist a tensor category $\C''$ and an exact sequence of tensor categories in the sense of Definition \ref{exact-rk1}: $$\C' \toto \C \toto \C''.$$ 

A tensor category $\C$ will be called \emph{simple} if it has no normal tensor subcategories with respect to a module category. If $\C$ has no normal tensor subcategories it will be called \emph{simple with respect to rank-one module categories}.

\subsection{Fusion subcategories of index 2 need not be normal}\label{indice2}

It is a well-known fact that if $G$ is a finite group, every subgroup of $G$ whose index is the smallest prime number dividing the order of $G$ is normal in $G$. More generally, if $H$ is a semisimple Hopf algebra and $H'$ is a Hopf subalgebra of $H$ such that $\dim H = p \dim H'$, where $p$ is the smallest prime number dividing the dimension of $H$, then $H'$ is normal in $H$ \cite{KM}. The next theorem gives a generalization of this result in the context of fusion categories. 

\medbreak 
Let $F:\C\to \D$ be a dominant tensor functor between finite tensor categories.
The \emph{Frobenius-Perron index of $F$} is defined as the ratio $\FPdim \C/ \FPdim\D$. The Frobenius-Perron index of a dominant tensor functor is an algebraic integer, by \cite[Corollary 8.11]{ENO}. In addition, if $G$ is a left (or right) adjoint of $F$, then the Frobenius-Perron index of $F$ coincides with $\FPdim G(\1)$. See \cite[Section 4]{tensor-exact}.

\begin{theorem}\label{index2}\emph{(\cite[Proposition 4.13]{tensor-exact}, \cite[Theorem 6.2]{indp-exact}).} Let $F: \C \to \D$ be a dominant tensor functor between fusion categories $\C$ and $\D$. Then the following hold:
	
\medbreak (i) If the Frobenius-Perron index of $F$ is 2, then $F$ is normal. 

\medbreak (ii) If $\C$ has integer Frobenius-Perron dimension and the Frobenius-Perron index of $F$ is the smallest prime number dividing $\FPdim \C$, then $F$ is normal.

\medbreak Furthermore, the exact sequences arising from (i) and (ii) are  equivariantization exact sequences. 
\end{theorem}

Let $\C$ be a finite tensor category and let $\C' \subseteq \C$ be a tensor subcategory. The ratio $\FPdim \C/\FPdim \C'$ will be called the \emph{index of $\C'$} (in $\C$). This is also an algebraic integer \cite{ENO}.

\medbreak  
The dual statement of Theorem \ref{index2} is not true, that is, there exist fusion subcategories of index 2 which are not normal. These examples are Tambara-Yamagami fusion categories $\mathcal{TY}(\mathbb Z_p, \chi, \tau)$ of  Frobenius-Perron dimension $2p$, where $p$ is a prime number.  In this case the pointed subcategory of $\mathcal{TY}(\mathbb Z_p, \chi, \tau)$, which is the unique fusion subcategory of Frobenius-Perron dimension $p$, is not normal \cite[Proposition 6.3]{indp-exact}. 

\medbreak 
In fact, Tambara-Yamagami categories $\mathcal{TY}(\mathbb Z_p, \chi, \tau)$ provide examples of fusion categories of dimension $2p$ which are simple. 
The next proposition generalizes \cite[Proposition 6.5]{indp-exact} to the context of exact sequences with respect to a module category.

\begin{proposition}\label{ty-simple} The fusion category $\mathcal{TY}(\mathbb Z_p, \chi, \tau)$ is simple. 
\end{proposition}

\begin{proof} Let $\C = \mathcal{TY}(\mathbb Z_p, \chi, \tau)$. Suppose on the contrary that there is  an exact sequence $\C' \toto \C \toto \C'' \boxtimes \End(\M)$, for some proper fusion subcategory $\C'$ and some indecomposable $\C'$-module category $\M$.  
Then $\FPdim \C = \FPdim \C' \FPdim \C''$ and therefore $\FPdim\C, \FPdim \C' \in \{2, p\}$. 

\medbreak 
Consider the associated exact factorization $\C^*_{\C'' \boxtimes \M} = \C'' \bullet (\C')^*_\M$. Since every fusion category of prime Frobenius-Perron dimension is pointed \cite{ENO}, then $\C''$, $(\C')^*_\M$ are pointed and therefore so is $\C^*_{\C'' \boxtimes \M}$. Thus $\C$ is group-theoretical. This is impossible because $\C$ is not integral (in fact the unique non-invertible simple object of $\C$ has Frobenius-Perron dimension $\sqrt{p}$). This contradiction shows that such an exact sequence cannot exist and therefore $\C$ is simple, as claimed.
\end{proof}

\subsection{Further examples of simple fusion categories of dimension $p^aq^b$}\label{simple-paqb}

Proposition \ref{ty-simple} implies that Burnside's $p^aq^b$-theorem does not extend to fusion categories in terms of exact sequences with respect to module categories. Further examples of this situation are provided by the non-group-theoretical fusion categories constructed in \cite{jl}. 

\medbreak 
Let $p < q$ be prime numbers such that $p$ is odd and divides $q + 1$. Let also $\zeta_1 \neq  \zeta_2 \in \mathbb F_{q^2}$ such that $\zeta_1^p = \zeta_2^p = 1$ but $\zeta_1\zeta_2 \neq 1$, and let $\xi \in H^3(\mathbb Z_p, k^\times) \cong \mathbb Z_p$. Consider the non-group-theoretical fusion category $\C(p, q, \{\zeta_1, \zeta_2\}, \xi)$ 
constructed in \cite{jl}. The fusion categories $\C(p, q, \{\zeta_1, \zeta_2\}, \xi)$ are $\Zz_p$-extensions of $\vect_{\Zz_q\times \Zz_q}$ and they fall into $(p^2 - p)/2$ equivalence classes. 

\medbreak 
We have $\FPdim \C(p, q, \{\zeta_1, \zeta_2\}, \xi) = pq^2$.

\begin{proposition} The fusion category $\C(p, q, \{\zeta_1, \zeta_2\}, \xi)$ is simple. 
\end{proposition}

\begin{proof} As in the proof of Proposition \ref{ty-simple}, let us assume on the contrary that there is  an exact sequence $\C' \toto \C \toto \C'' \boxtimes \End(\M)$, for some proper fusion subcategory $\C'$ and some indecomposable $\C'$-module category $\M$.  
So that $\FPdim \C''$ can be either $p$, $q$, $q^2$ or $pq$ and similarly for $\FPdim \C'$. Therefore both $\C''$ and $(\C')^*_\M$ are pointed: this follows from \cite[Corollaries 8.30 and 8.31]{ENO} if the Frobenius-Perron dimensions are $p$, $q$ or $q^2$, and by the classification of fusion categories of dimension $pq$ in \cite[Theorem 6.3]{ego}, since $p$ is odd (and thus it cannot divide $q-1$). 

\medbreak 
This implies that  $\C^*_{\C'' \boxtimes \M} = \C'' \bullet (\C')^*_\M$ is pointed and then $\C$ is group-theoretical. This contradicts the choice of $\C$ and shows that such an exact sequence cannot exist. Thus $\C$ is simple, as claimed.
\end{proof}

\subsection{Group-theoretical fusion categories and exact sequences}
As mentioned before, the category of representations of a finite simple group is a simple fusion category. 

Notice, however, that there exist simple Hopf algebras $H$ such that the tensor category $H\modd$ is not simple; see Subsection \ref{subs-simple}.

\medbreak 
Recall that a fusion category $\C$ is called {\it group-theoretical} if it is
categorically Morita equivalent to a pointed fusion category.  
Let $\C$ be a pointed fusion category, so that there exist a finite group $G$
and a
3-cocycle $\omega:G \times G \times G \to k^{\times}$ such that $\C$ is equivalent to the category $\vect_G^\omega$ of $G$-graded vector spaces with associativity determined by $\omega$.

\medbreak
Every indecomposable module category over $\vect_G^\omega$ arises from a pair
$(\Gamma, \alpha)$, where $\Gamma$ is a subgroup
of $G$ and $\alpha: \Gamma \times \Gamma \to k^{\times}$ is a 2-cochain on $\Gamma$ such that $d\alpha=\omega|_{\Gamma\times \Gamma\times \Gamma}$. Thus, the restriction
$\omega\vert_\Gamma$ represents the trivial cohomology class in $H^3(\Gamma, k^\times)$. Given such a pair $(\Gamma, \alpha)$, the twisted group algebra $A(\Gamma, \alpha) = k_\alpha \Gamma$ is an indecomposable algebra in $\vect_G^\omega$.
The (left) module category associated to such pair $(\Gamma, \alpha)$ is the category
$$\M(\Gamma, \alpha) = (\vect_G^\omega)_{A(\Gamma, \alpha)}$$ of (right)  $A(\Gamma, \alpha)$-modules in $\vect_G^\omega$. 

\medbreak 
The group-theoretical category $(\vect_G^\omega)^*_{\M_0(\Gamma, \alpha)}$ is
denoted $\C(G, \omega, \Gamma, \alpha)$. 

\medbreak Let $G$ be a finite group and let $\Gamma$ be a subgroup of $G$. 
There is a canonical embedding of tensor categories $\Rep \Gamma \toto \C(G, \omega, \Gamma, \alpha)$.
The next proposition implies that $\Rep \Gamma$ is not necessarily a normal tensor subcategory of $\C(G, \omega, \Gamma, \alpha)$.

\medbreak 
Let $G$ be the alternating group $\Aa_6$ of order $60$ and let $\Gamma$ be a subgroup of $G$ such that $G \cong \Aa_5$. Then  $\Gamma$ is a maximal subgroup of $\Aa_6$. There are 12 such subgroups and they constitute 2 conjugacy classes, represented by the subgroups $\langle (12345), (123)\rangle$ and $\langle (1,2,3,4,5), (1,4)(5,6) \rangle$. 	

\medbreak 	
Let $\C = \C(\Aa_6, 1, \Gamma, 1)$, so that $\C$ is a group-theoretical fusion category of dimension $360$ which is categorically Morita equivalent to $\vect_{\Aa_6}$. In addition, $\C$ contains a fusion subcategory $\C'\cong \Rep \Aa_5$.

\medbreak 	
Observe that $\C$ is not a graded extension of any fusion category: this follows from the characterization of graded extensions in  \cite[Proposition 2.9]{ENO2}, since $\Rep \Aa_6$ is the unique Tannakian subcategory of $\Z(\C)$ (c.f. \cite[Example 3.3]{core-wgt}) and $\C$ cannot be an $\Aa_6$-graded extension of a fusion subcategory, since it is not pointed.

\medbreak 	
We have $\widehat \Gamma = 1$ and $\Gamma = N_{\Aa_6}(\Gamma)$ (by maximality of $\Gamma$). Hence, by \cite[Theorem 5.2]{gel-nai}, the group of invertible objects of $\C$ is trivial.	

\begin{proposition}\label{a5-notnormal} The category $\C(\Aa_6, 1, \Gamma, 1)$ is simple with respect to rank-one module categories. 	
\end{proposition}	

\begin{proof}
Suppose that $F: \C \to \D$ is a normal dominant tensor functor such that $F$ is not an equivalence. Then the restriction of $F$ to $\C' \cong  \Rep \Aa_5$ is also a normal tensor functor with kernel $\KER_F \cap \C'$. Therefore either $\C' \cap \KER_F \cong \vect$ or $\C' \subseteq \KER_F$. Suppose first that $\C' \cap \KER_F \cong \vect$. Then the restriction of $F$ to $\C'$ is a full embedding, whence the dimension of $\D$ is divisible by $60$. Then the dimension of $\KER_F$ can equal $6$, $2$ or $3$. But this contradicts the fact that $\C$ has no nontrivial invertible objects. Hence $\C' \subseteq \KER_F$. 
	
The Grothendieck ring $\operatorname{Gr} (\C)$ of $\C$ is faithfully graded by the double coset ring $\Gamma \backslash \Aa_6 / \Gamma$. In addition the fusion subcategory $\KER_F$ determines a based subring of $\operatorname{Gr} (\C)$. Since $\C' \subseteq \KER_F$, this based subring corresponds to a subgroup of $\Aa_6$ containing $\Gamma$, see  \cite[Theorem 4.1 and Lemma 3.6]{gel-nai}. The maximality of $\Gamma$ implies that $\C' = \KER_F$. Then $\C$ fits into an exact sequence $\C' \toto \C \toto \D$, where $\D$ is a fusion category of dimension $6$. Moreover, by \cite[Proposition 4.9]{tensor-exact}, $\D$ is integral.
	
	\medbreak
	Assume that the exact sequence $\C' \toto \C \toto \D$ is \emph{abelian}. In other words the sequence is exact with respect to the trivial $\C'$-module category $\vect$ in the sense of \cite{eg-emc}. Let $(A, \sigma) \in \Z(\C)$ be the induced central algebra of $F$.
	Then ${}_A\C_A$ has a factorization into fusion subcategories equivalent, respectively, to $\vect_{\Gamma}$ and $\D$. If $\D$ is pointed of dimension $6$, then $\D \cong \vect^\omega_{S}$, where $S$ is a group of order $6$. Hence, ${}_A\C_A \cong \vect^{\omega'}_L$, where $L$ is a group of order $360$ that has an exact factorization $L = \Gamma . H$ \cite[Proposition 7.3]{mn}. In particular the pointed fusion category $\vect^{\omega'}_L$ is categorically Morita equivalent to $\vect_{\Aa_6}$. By \cite[Theorem 5.8]{naidu}, this implies that $L \cong \Aa_6$. This leads to a contradiction, because the group $\Aa_6$ admits no exact factorization into proper subgroups \cite{miller}.
	
	\medbreak
	If, on the other hand, $\D$ is not pointed, then $\D$ must contain two distinct invertible objects and a simple object of dimension $2$. Then the group of invertible objects of ${}_A\C_A$ is of order $120$. Since ${}_A\C_A$ is Morita equivalent to $\vect_{\Aa_6}$, then there exists a subgroup $T$ of $\Aa_6$ and a 2-cocycle $\psi$ on $T$ such that ${}_A\C_A \cong \C(\Aa_6, 1, T, \psi)$. By \cite[Theorem 5.2]{gel-nai}, the group of invertible objects of ${}_A\C_A$ has order $|K| |\widehat T|$, where $K$ is a certain subgroup of $N_{\Aa_6}(T)/T$. A direct inspection on the possible subgroups $T$ of $\Aa_6$ (see Table \ref{sps-a6}) shows that $|N_{\Aa_6}(T)/T| |\widehat T| \leq 36$, which is again a contradiction.
	
	\medbreak 
	We have thus shown that the group-theoretical fusion category $\C = \C(\Aa_6, 1, G, 1)$ does not fit into any \emph{abelian} exact sequence of tensor categories.
	
	\medbreak Therefore, if $\C$ is not simple, then $\C$ fits into a non-abelian exact sequence 
	\begin{equation}\label{ex-a6}\C' \toto \C \toto \D, 
	\end{equation} 
	where $\C' \cong \Rep \Aa_5$ and $\FPdim \D = 6$. By the previous part, the induced Hopf algebra $H$ of \eqref{ex-a6} is not commutative. Hence $H \cong (k\Aa_5)^J$, where the twist $J$ is not trivial. Then $H\textrm{-mod}$ is of type $(1, 12; 4, 3)$, see \cite{sixty}.
	
	\medbreak 
	Consider the associated exact factorization $\C^*_\D \cong H\textrm{-mod} \bullet \D$. 
	As before, the fusion category $\D$ is either pointed or of type $(1, 2; 2, 1)$. Hence the possible types for $\C^*_\D$ are 
	$$\textrm{(i) } \; (1, 72; 4, 18), \qquad \textrm{(ii) } \; (1, 24; 2, 12; 4, 6; 8, 3).$$ 
	
	Write, as before, ${}_A\C_A \cong \C(\Aa_6, 1, T, \psi)$, where $T$ is a subgroup of $\Aa_6$ and $\psi$ is a 2-cocycle on $T$. As pointed out before, the group of invertible objects of $\C(\Aa_6, 1, T, \psi)$ is at most 36, hence possibility (i) is discarded. 
	
	\medbreak 		
	If possibility (ii) holds, then $T$ must be a Klein four group (Table \ref{sps-a6}). Observe that, for every simple object $X$ of $\C$, $\FPdim X$ divides $|T|$; see \cite[Proposition 5.5]{pjm}. 
	This contradicts the fact that $\C^*_\D$ has simple objects of dimension 8 as in case (ii).  This shows that $\Rep \Gamma$ is not normal in $\C = \C(\Aa_6, 1, \Gamma, 1)$ and therefore $\C$ is simple, as was to be shown. 
\end{proof}

\subsection{Dominant images of normal fusion subcategories need not be normal}\label{dom-img}
Proposition \ref{a5-notnormal} provides examples of images of normal fusion subcategories under dominant tensor functors which are not normal. 

\medbreak 
Let $G \cong \Aa_6$ and $\Aa_5 \cong \Gamma \subseteq \Aa_6$ as in the previous subsection. Let also $\C = \C(G, 1, \Gamma, 1)$.

\medbreak 
The canonical braiding of $\Rep G$ gives rise to an embedding of braided tensor categories $\Rep G \toto \Z(G)$ that fits into an equivariantization exact sequence of tensor categories
$$\Rep G \toto \Z(G) \toto \vect_G.$$

On the other hand, since $\C$ is Morita equivalent to $\rep G$, there is an  equivalence of braided tensor categories $\Z(G) \cong \Z(\C)$ \cite{schauenburg}. This equivalence induces a dominant tensor functor $$U_\Gamma: \Z(G) \toto \C,$$ such that $U_\Gamma(V) = k\Gamma \otimes V$. 
In particular $U_\Gamma (\Rep G) = \Rep \Gamma$. Since, by Proposition \ref{a5-notnormal}, $\C$ is simple this implies:

\begin{corollary}The image of the normal fusion subcategory $\Rep G$ of $\Z(G)$ under the dominant tensor functor $U_\Gamma: \Z(G) \to \C$ is not normal in $\C$.
\end{corollary} 

\subsection{Composition series and composition factors}\label{jh-wrtm}

The definition of a composition series of a finite dimensional Hopf algebra (Definition \ref{def-compser}) has an obvious extension to the context of exact sequences of finite tensor categories with respect to module categories. 

\medbreak Let a \emph{composition series} of a finite tensor category $\C$ be defined as a
sequence of finite tensor categories $\C_1, \dots, \C_n$
defined, as before, as  $n = 1$ and $\C_1 = \C$,
if $\C$ is simple, while if $\C' \toto \C \toto \C'' \boxtimes \End(\M)$ is an exact sequence with respect to the $\C'$-module category $\M$ such that $\FPdim \C', \FPdim \C'' > 1$, and
$\C'_1,
\dots, \C'_m$, $\C''_1, \dots,$ $\C''_\ell$, are composition series of $\C'$ and $\C''$,
respectively, then $n = m+\ell$ and
$$\C_i = \begin{cases} \C'_i, \qquad 1\leq i \leq m, \\
\C''_{i-m}, \quad  m < i \leq m+\ell.\end{cases}$$
As before, the \emph{factors} and the \emph{length} of the series are, respectively, the finite tensor categories $\C_1, \dots, \C_n$ and the number $n$.

\medbreak 
It is clear that every finite tensor category admits such a composition series.

\begin{example}Let $G$ be a finite simple group. Then the fusion category $\Rep G$  is simple: in fact, $\Rep G$ has no proper fusion subcategories when $G$ is simple. Observe that this is not true for the category $\vect_G$,  since the finite simple group $G$ might have non-trivial exact factorizations into proper subgroups. In particular the condition on a fusion category being simple is not self-dual.
	
	\medbreak 	
	On the other hand, if $H$ is a normal subgroup of $G$, then the restriction functor gives rise to an exact sequence in the sense of \cite{tensor-exact}:
	$$\Rep G/H \toto \Rep G \toto \Rep H.$$
	Inductively, we find that if $G_1, \dots, G_n$ are the composition factors of $G$, then the fusion categories 
	$$\Rep G_1, \dots, \Rep G_n,$$ 
	are composition factors of $\Rep G$. 
\end{example}

\medbreak 
In what follows we show that  composition series of fusion categories thus defined fail to satisfy a Jordan-H\"older theorem.

\medbreak Let $n \in \mathbb{N}$. We shall denote by $\Ss_n$ and $\Aa_n$ the symmetric and alternating groups of degree $n$, respectively. 

\medbreak 
The proof of the following theorem relies on a result of Miller \cite{miller} that asserts that the alternating group $\Aa_6$ does not admit any exact factorization into proper subgroups. We summarize in Tables \ref{sps-a6} and \ref{sps-a5} the information about the subgroups of $\Aa_6$ and $\Aa_5$ used along the proof.

\begin{theorem}\label{a6-simple} The fusion category $\vect_{\Aa_6}$ is simple. \end{theorem}

\begin{proof} Suppose on the contrary that there exists an exact sequence with respect to an indecomposable $\C'$-module category $\M$:
	\begin{equation}\label{exact-a6}\C'\toto \vect_{\Aa_6} \overset{F}\toto \C'' \boxtimes \End(\M),
	\end{equation}
	such that $\FPdim \C', \FPdim \C'' > 1$.
	
Then $\C'$ and $\C''$ are fusion categories. 
Since the functor $F$ is dominant then $\C''$ is pointed, by Lemma \ref{quotient-pt}.
Also $\C'$ is pointed; moreover, $\C' = \vect_{H}$, for some subgroup $H$ of $\Aa_6$ such that $1 \subsetneq H \subsetneq \Aa_6$.
Hence $\M = \M(T, \psi) = \C'_{A(T, c)}$, for some subgroup $T$ of $H$ and some 2-cocycle $\psi$ on $T$.
In addition, $360 = \FPdim \vect_{\Aa_6} = \FPdim \C'  \FPdim \C'' = |H| \FPdim \C''$.
	
	\medbreak Consider the associated exact factorization of the dual fusion category $\D = (\vect_{\Aa_6})^*_{\C'' \boxtimes \End(\M)} \cong  {}_{A(T, \psi)}(\vect_{\Aa_6})_{A(T, \psi)} \cong \C(\Aa_6, 1, T, \psi)$:
	\begin{equation}
	\D = \C(\Aa_6, 1, T, \psi) = \C(H, 1, T, \psi) \bullet \C''.
	\end{equation}
	
	Since $\Aa_6$ admits no exact factorizations, then \cite[Theorem 3.4]{naidu} implies that $\D$ is not pointed. 
	Also, since $\FPdim \C'' > 1$, then $\FPdim \D_{pt} > 1$.
	Moreover $\C''$ is contained in $\D_{pt}$ and $\FPdim \C'' = [\Aa_6: H]$. Also  $[\Aa_6: H]$ divides $[\Aa_6:T]$ because $T$ is a subgroup of $H$.
	
	\medbreak 
	By \cite[Theorem 5.2]{gel-nai}, the group of invertible objects of $\D$ is an extension of $\widehat T$ by a certain subgroup $K$ of $N_{\Aa_6}(T)/T$.
	
	Then  $\FPdim \D_{pt}$ divides $|N_{\Aa_6}(T)/T| |\widehat T|$, and therefore  $[\Aa_6: H] | |N_{\Aa_6}(T)/T| |\widehat T|$.
	Thus $[\Aa_6: H]$ divides $\textrm{gcd} ([\Aa_6:T], |N_{\Aa_6}(T)/T| |\widehat T|)$. A direct inspection on the possible subgroups $T$ in Table \ref{sps-a6}, combined with the fact that $\Aa_6$ has no subgroups of index $2$, $3$ or $4$, implies that $[\Aa_6: H] = 6$ and $H \cong \Aa_5$. 
	
	\medbreak In addition the subgroup $T  \cong \Zz_2\times \Zz_2$ is a Klein four group, or $T \cong \Zz_3$, or $T \cong \Aa_4$. 
	
	\medbreak 
	Suppose first that $T \cong \Zz_2 \times \Zz_2$. Then $\FPdim \D_{pt}$ divides $|N_{\Aa_6}(T)/T| |\widehat T| = 24$. In addition $|\widehat T|$ divides $\FPdim \C(H, 1, T, \psi)_{pt}$ and therefore $24$ divides $\FPdim \D_{pt}$. Hence $\FPdim \D_{pt} = 24 = |N_{\Aa_6}(T)/T| |\widehat T|$. This implies that the subgroup $K$ in \cite[Theorem 5.2]{gel-nai} coincides with $N_{\Aa_6}(T)/T$.
	
	From the definition of $K$ in \cite{gel-nai}, we find that for every $g \in N_{\Aa_6}(T)$, the class of the 2-cocycle $\psi^g: T \times T \to k^\times$, defined by 
	$$\psi^g(h_1, h_2) = \psi(h_1, h_2) \, \psi(g^{-1}h_2^{-1}g, g^{-1}h_1^{-1}g), \quad h_1, h_2 \in T,$$  is trivial. 
	Hence the cocycle $\psi^g$ is also trivial for all $g \in N_H(T)$ and therefore $\FPdim \C(H, 1, T, \psi)_{pt} = |N_{H}(T)/T| |\widehat T| = 12$, since the group $N_{H}(T)/T$ is cyclic of order 3 (see Table \ref{sps-a5}).
	
	Combined with the fact that $\FPdim \C'' = 6$ and $\C''$ is pointed, this implies that $\FPdim \D_{pt} = 72$, and we arrive to a contradiction. Thus we have discarded this possibility for $T$.
	
	\medbreak 
	Assume next that $T \cong \Zz_3$. In this case the class of $\psi$ is trivial and $N_{\Aa_6}(T)/T$ is of order $6$. Thus the subgroup $K$ in \cite[Theorem 5.2]{gel-nai} coincides with $N_{\Aa_6}(T)/T$ and $\FPdim \D_{pt} = 18$. On the other hand, $N_{H}(T)/T$ is of order $2$ (Table \ref{sps-a5}) and therefore $\FPdim \C(H, 1, T, \psi)_{pt} = 6$. The exact factorization \eqref{exact-a6} then implies that $\FPdim \D_{pt} = 36$, which is a contradiction. Then this possibility is also discarded. 
	
	\medbreak Finally suppose that $T \cong \Aa_4$. In this case $\widehat T$ is of order $3$ and $N_{\Aa_6}(T)/T$ is of order $2$. Thus $\FPdim \D_{pt}$ divides $6$. But since $|\widehat T|$ also divides $\FPdim \C(H, 1, T, \psi)_{pt}$, and $\C''$ is pointed then $9$ divides $\FPdim \D_{pt}$, which is impossible.
	This discards this possibility as well. Thus we have shown that such an exact factorization \eqref{exact-a6} cannot exist and therefore $\vect_{\Aa_6}$ is a simple fusion category, as claimed.
\end{proof}

\begin{corollary}\label{cs-a6} The composition factors of a fusion category may be non-unique up to permutation. Moreover, a fusion category may admit composition series with different length.
\end{corollary}

\begin{proof}
	The group $\Ss_6$ has an exact factorization $\Ss_6 = \Aa_6 \bullet \Zz_2$. This induces an exact factorization $\vect_{\Ss_6} = \vect_{\Aa_6} \bullet \vect_{\Zz_2}$. Therefore, from \cite[Theorem 4.1]{gelaki}, there is an exact sequence with respect to $\vect_{\Aa_6}$:
	\begin{equation}\label{first}\vect_{\Aa_6} \toto \vect_{\Ss_6} \toto \vect_{\Zz_2} \boxtimes \End(\vect_{\Aa_6}).
	\end{equation}
	
	By Theorem \ref{a6-simple}, $\vect_{\Aa_6}$ is simple. Then \eqref{first} gives a composition series for $\vect_{\Ss_6}$ with factors $\vect_{\Aa_6}$, $\vect_{\Zz_2}$.
	
	\medbreak 
	On the other hand, the exact factorizations
	\begin{align*}
	\Ss_6 & = \Ss_5 \bullet \Zz_6 \\
	& = \Ss_4 \bullet\Zz_5 \bullet \Zz_3 \bullet \Zz_2 \\
	& = \Ss_3 \bullet \Zz_4 \bullet \Zz_5 \bullet \Zz_3 \bullet \Zz_2 \\
	& = \Zz_3 \bullet \Zz_2 \bullet \Zz_4 \bullet \Zz_5 \bullet \Zz_3 \bullet \Zz_2\\
	& = \Zz_3 \bullet \Zz_2 \bullet \Zz_4  \bullet \Zz_5 \bullet \Zz_3 \bullet \Zz_2,
	\end{align*}
	induce exact factorizations of the corresponding pointed fusion categories. An iterated application of \cite[Theorem 4.1]{gelaki} implies that $\vect_{\Ss_6}$ has a composition series with factors $\vect_{\Zz_3}$, $\vect_{\Zz_2}$, $\vect_{\Zz_2}$, $\vect_{\Zz_2}$, $\vect_{\Zz_5}$, $\vect_{\Zz_3}$, $\vect_{\Zz_2}$. Then we see that the factors, and also  the length, of these composition series are not unique. 
\end{proof}

\begin{remark}\label{bd-cs-a6} The statement of Corollary \ref{cs-a6} remains valid when restricted to the class of non-degenerate braided fusion categories, namely, also in this case the compositions factors may  be non-unique up to permutation and composition series may have different lengths.
	
	\medbreak As an example, consider the Drinfeld center $\Z(\vect_{\Ss_6})$. 
	Thus $\Z(\vect_{\Ss_6})$ coincides with the representation category of the Drinfeld double $D(k^{\Ss_6})$ and there is an exact sequence of finite dimensional Hopf algebras 
	\begin{equation*}k \toto k^{\Ss_6} \toto D(k^{\Ss_6}) \toto k{\Ss_6} \toto k.
	\end{equation*}
	Therefore the non-degenerate braided fusion category $\Z(\vect_{\Ss_6})$ fits into an exact sequence in the sense of \cite{tensor-exact}:
	\begin{equation*} \Rep{\Ss_6} \toto \Z(\vect_{\Ss_6}) \toto \vect_{\Ss_6}.
	\end{equation*}	
	In addition we have an exact sequence
	$$\Rep{\Zz_2} \toto \Rep{\Ss_6} \toto \Rep{\Aa_6},$$
	and since $\Rep{\Aa_6}$ is a simple fusion category, the composition factors of $\Rep{\Ss_6}$ are $\Rep{\Zz_2}$ and $\Rep{\Aa_6}$. 
	
	\medbreak 	
	The proof of Corollary \ref{cs-a6} gives two composition series of $\vect_{\Ss_6}$ that give rise to two composition series of $\Z(\vect_{\Ss_6})$ with factors
	$$\Rep{\Zz_2}, \Rep{\Aa_6}, \vect_{\Aa_6}, \vect_{\Zz_2},$$
	on the one hand, and
	$$\Rep{\Zz_2}, \Rep{\Aa_6}, \vect_{\Zz_3}, \vect_{\Zz_2}, \vect_{\Zz_2}, \vect_{\Zz_2}, \vect_{\Zz_5}, \vect_{\Zz_3}, \vect_{\Zz_2}$$
	on the other hand. Then the series have different length and the factors are not unique also in this case. 
\end{remark}

\begin{table}[t]
	\begin{center}
		\begin{tabular}{|p{6cm}|p{2,5cm}|p{0,5cm}|p{0,5cm}|p{1,8cm}|}
			\hline  {\bf \quad \newline Automorphism class \newline representative} & {\bf \quad \newline Isomorphism class of $T$ } & {\bf \quad \newline $|T|$ \newline \quad} & {\bf \quad \newline $|\widehat{T}|$ \newline \quad } & {\bf \quad \newline $[N_G(T):T]$ \newline \quad }  \\ 
			\hline \{e\} & Trivial & 1 & 1   & 360 \\
			\hline  $\langle (12)(34) \rangle$ & $\Zz_2$ & 2 & 2  &  4\\ 
			\hline  $\langle (12)(34), (13)(24) \rangle$ & $\Zz_2 \times \Zz_2$ & 4 & 4  &6 \\ 
			\hline $\langle (12)(34), (12)(56) \rangle$ & $\Zz_2 \times \Zz_2$ & 4 & 4 & 6 \\ 
			\hline $\langle (1234)(56) \rangle$ & $\Zz_4$ & 4 & 4 & 2\\ 
			\hline $\langle (1234)(56), (13)(56) \rangle$  & $D_4$ & 8 & 4 & 1 \\ 
			\hline $\langle (123)\rangle$ & $\Zz_3$ & 3 & 3 & 6\\ 
			\hline $\langle (123)(456)\rangle$ & $\Zz_3$ & 3 & 3  & 6 \\
			\hline $\langle (123), (456)\rangle$ & $\Zz_3 \times \Zz_3$ & 9 &  9 & 4\\ 
			\hline $\langle (123)(456), (12)(45)\rangle$ & $\Ss_3$ & 6 & 2 & 1\\ 
			\hline $\langle (123), (12)(45)\rangle$ & $\Ss_3$ & 6 & 2 & 1\\ 
			\hline $\langle (12)(34), (123)\rangle$ & $\Aa_4$ & 12 & 3 & 2				-- \\ 
			\hline $\langle (123)(456), (14)(25), (14)(36)\rangle$ & $\Aa_4$ & 12 & 3 & 2 \\ 
			\hline $\langle (1234)(56), (12)(56)\rangle$ & $\Ss_4$ & 24 & 2 & 1 \\ 
			\hline $\langle (34)(56), (12)(56), (135)(246), (35)(46)\rangle$ & $\Ss_4$ & 24 & 2 & 1 \\
			\hline $\langle (123), (456), (12)(45)\rangle$ & $(\Zz_3\times\Zz_3)\rtimes \Zz_2$ & 18 & 2 & 2 \\ 
			\hline $\langle (123), (456), (23)(56), (14)(2536)\rangle$ & $(\Zz_3\times\Zz_3)\rtimes \Zz_4$ & 36 & 4  & 1 \\ 
			\hline $\langle (12345)\rangle$ & $\Zz_5$ & 5 & 5 & 2 \\ 
			\hline $\langle (12345), (25)(34)\rangle$ & $D_5$ & 10 & 2 & 1 \\ 
			\hline $\langle (12345), (123)\rangle$ & $\Aa_5$ & 60 & 1 &	1\\ 
			\hline $\langle (12345), (14)(56)\rangle$ & $\Aa_5$ & 60 & 1 & 1 \\ 
			\hline Whole group & $\Aa_6$ & 360 & 1 & 1\\ 
			\hline
		\end{tabular}
	\end{center}
	
	\
	
	\caption{Subgroups $T$ of the alternating group $\Aa_6$}\label{sps-a6}
\end{table}

\begin{table}[t]
	\begin{center}
		\begin{tabular}{|p{6cm}|p{2,5cm}|p{1,8cm}|}
			\hline  {\bf \quad \newline Automorphism class \newline representative} & {\bf \quad \newline Isomorphism class of $T$ } &  {\bf \quad \newline $[N_G(T):T]$ \newline \quad }  \\ 
			\hline \{e\} & Trivial &  60 \\
			\hline  $\langle (12) \rangle$ & $\Zz_2$ &  2\\ 
			\hline  $\langle (12)(34) \rangle$ & $\Zz_2 \times \Zz_2$ &  3 \\ 
			\hline $\langle (123)\rangle$ & $\Zz_3$  & 2 \\ 
			\hline $\langle (123), (12)(45) \rangle$ & $\Ss_3$  & 1\\ 
			\hline $\langle (12)(34), (123)\rangle$  & $\Aa_4$  & 1 \\ 
			\hline $\langle (12345)\rangle$ & $\Zz_5$  & 2\\ 
			\hline $\langle (12345), (25)(34)\rangle$ & $D_5$  & 1 \\ 
			\hline Whole group & $\Aa_5$  & 1\\ 
			\hline
		\end{tabular}
	\end{center}
	
	\
	
	\caption{Subgroups $T$ of the alternating group $\Aa_5$}\label{sps-a5}
\end{table}

\subsection{Questions}\label{preguntas}

We think it is interesting to determine classes of finite tensor categories which are closed under extensions. For instance, the class of fusion categories is closed under extensions and so is the class of weakly integral and integral finite tensor categories. 

\medbreak 
On the other hand, it is known that the class of group-theoretical fusion categories is not closed under extensions: Indeed, let $p$ be an odd prime number and let $H$ be one of the non-group-theoretical semisimple Hopf algebras of dimension $4p^2$ constructed by D. Nikshych in  \cite{nik}. Then $H$ fits into an exact sequence of Hopf algebras
$$k \toto k^{\Zz_2} \toto H \toto A_p \toto k,$$
where $A_p$ is a certain abelian extension of $\Zz_2$ by $\Zz_p \times \Zz_p$.

Since every abelian extension is group-theoretical\footnote{A Hopf algebra $H$ is called group-theoretical if the category $H\modd$ (or equivalently, the category $\CoRep H$)  is group-theoretical.} \cite{gp-ttic}, then the Hopf algebra $A_p$
is group-theoretical. Then $H\modd$ is a non-group-theoretical fusion category that fits into an exact sequence  of group-theoretical fusion categories
$$A_p\modd \toto H\modd \toto \vect_{ \Zz_2}.$$
This exact sequence is in fact an equivariantization exact sequence, by construction of $H$. See \cite[Section 5]{nik}.

\medbreak 
A fusion category $\C$ is called weakly group-theoretical if it is Morita equivalent to a nilpotent fusion category \cite{ENO2}. The class of weakly group-theoretical fusion categories is known to be closed under a number of operations, like taking a fusion subcategory or dominant image, taking Deligne tensor product and Drinfeld center. It is not known if it closed under extensions: 

\begin{question} \emph{(\cite[Question 4.9]{gelaki}.)} Is the class of weakly group-theoretical fusion categories closed under extensions?
\end{question} 

It is known that the class of weakly group-theoretical fusion categories is closed under equivariantizations and group graded extensions \cite[Proposition 4.1]{ENO2}. 
Furthermore, it also closed under matched pair crossed extensions \cite[Corollary 4.6]{char-crossed-action}. 

\medbreak 
Let $G$ be a finite group. In \cite{char-crossed-action} we showed that if $\C$ is a fusion category fitting into an abelian exact sequence $\Rep G \toto \C \toto \D$ or into an exact sequence $\vect_G \toto \C \toto \D$, then $\C$ is  weakly group-theoretical if and only if $\D$ is weakly group-theoretical.
As a consequence, every  semisolvable semisimple Hopf algebra, as introduced in \cite{MW}, is weakly group-theoretical.  

\medbreak 
Recall that a fusion category is said to satisfy the \emph{Frobenius property} if the ratio $\FPdim \C / \FPdim X$ is an algebraic integer, for every simple object $X$ of $\C$. It is known that every pre-modular fusion category satisfies the Frobenius property \cite{ENO2}. 

\begin{question}\label{q-frobppty} Is the class of fusion categories with the Frobenius property closed under extensions?
\end{question}

It follows from \cite[Theorem 1.5]{ENO2} that the class of fusion categories with the Frobenius property is closed under equivariantizations and group graded extensions, and every weakly group-theoretical fusion category satisfies the Frobenius property. The answer to  Question \ref{q-frobppty} is not known in general even in the context of Hopf algebra extensions. 

\medbreak 
Another interesting class of tensor categories is that of \textit{Frobenius tensor categories}: these are tensor categories in which every simple object has an injective hull (equivalently, a projective cover) \cite[Subsection 2.3]{AEC}. For instance, finite tensor categories and semisimple tensor categories are Frobenius categories. 

\medbreak 
Examples of Frobenius categories are provided by the categories of finite dimensional comodules over co-Frobenius Hopf algebras: that is, Hopf algebras $H$ endowed with a nonzero integral $H \to k$. 

It is known that every Hopf algebra $H$ fitting into a strictly exact sequence of Hopf algebras $k \toto H' \toto H \toto H'' \toto k$, such that $H'$ and $H''$ are  co-Frobenius Hopf,  is co-Frobenius \cite[Theorem 2.10]{AC}. This result allows to construct examples of this kind of Hopf algebras from smaller examples. We do not known the answer to the corresponding question for tensor categories:

\begin{question} Is the class of Frobenius tensor categories closed under extensions?
\end{question}

Regarding the notion of simplicity of a finite tensor category, the following is a natural question, about which very little is known: 

\begin{question} Is it possible to classify simple (finite) tensor categories? 
\end{question}

Finally, motivated by the examples of Subsection \ref{jh-wrtm}, we ask:

\begin{question} Can the definition of a composition series be reformulated in order that the Jordan-H\" older theorem holds for finite tensor categories?
\end{question}

An analogue of the Jordan-H\" older theorem was proved in \cite{jh-wgt} for weakly group-theoretical fusion categories. The definition of a composition series for this kind of category is given in terms of group equivariantizations and group graded extensions and the composition factors, which are Morita invariants, are simple finite groups.

\bibliographystyle{amsalpha}

\end{document}